\documentclass[12pt,twoside,draft,reqno]{amsart}\usepackage[cp1251]{inputenc}%\usepackage[english,russian]{babel}
\usepackage{srcltx}%\usepackage{refcheck}
\usepackage{amssymb}

\advance\textwidth30mm \advance\hoffset-15mm
\advance\textheight40mm \advance\voffset-20mm
\allowdisplaybreaks[2]
\theoremstyle{plain}
\newtheorem{theorem}{Theorem}[section]
\newtheorem{corollary}[theorem]{Corollary}
\newtheorem{proposition}[theorem]{Proposition}
\newtheorem{lemma}[theorem]{Lemma}
\theoremstyle{remark}
\newtheorem{remark}{Remark}[section]
\numberwithin{equation}{section}

\newcommand\Cdot{{\mskip2mu\cdot\mskip2mu}}

\begin{document}

\title[Sharp approximation theorems and Fourier inequalities]
{Sharp approximation theorems and Fourier inequalities
in the Dunkl setting}

\author{D.~V.~Gorbachev}
\address{D.~V.~Gorbachev, Tula State University,
Department of Applied Mathematics and Computer Science,
300012 Tula, Russia}
\email{dvgmail@mail.ru}

\author{V.~I.~Ivanov}
\address{V.~I.~Ivanov, Tula State University,
Department of Applied Mathematics and Computer Science,
300012 Tula, Russia}
\email{ivaleryi@mail.ru}

\author{S.~Yu.~Tikhonov}
\address{S.~Yu.~Tikhonov, Centre de Recerca Matem\`atica\\
Campus de Bellaterra, Edifici~C 08193 Bellaterra (Barcelona), Spain; ICREA, Pg.
Llu\'is Companys 23, 08010 Barcelona, Spain, and Universitat Aut\`onoma de
Barcelona.}
\email{stikhonov@crm.cat}

\date{\today}

\thanks{The work of D.~Gorbachev and V.~Ivanov is supported by the Russian Science
Foundation under grant 18-11-00199 and performed in Tula State University.
S.~Tikhonov was partially supported by MTM 2017-87409-P, 2017 SGR 358, and the
CERCA Programme of the Generalitat de Catalunya.}

\begin{abstract}
In this paper we study direct and inverse approximation inequalities in
$L^{p}(\mathbb{R}^{d})$, $1<p<\infty$, with the Dunkl weight. We obtain these
estimates in their sharp form substantially improving previous results. We
also establish new estimates of the modulus of smoothness of a function $f$ via
the fractional powers of the Dunkl Laplacian of approximants of $f$. Moreover,
we obtain new Lebesgue type estimates for moduli of smoothness in terms of
Dunkl transforms. Needed Pitt-type and Kellogg-type Fourier--Dunkl inequalities
are derived.
\end{abstract}

\maketitle
%\tableofcontents

\section{Introduction}%\label{sec0}

\subsection{Notation}
{Throughout the paper, $(x,y)$ denotes the scalar product in the
$d$-dimensional Euclidean space $\mathbb{R}^d$, $d\in \mathbb{N}$, equipped
with a norm $|x|=\sqrt{(x,x)}$. By  $B_r(x_{0})=\{x\in\mathbb{R}^d\colon
|x-x_{0}|\le r\}$ we denote  the Euclidean ball.}

Let a finite subset $R\subset \mathbb{R}^{d}\setminus\{0\}$ be a root system and
$R_{+}$ be a positive subsystem of $R$.
By $G(R)\subset O(d)$ we denote a finite reflection group,
generated by reflections $\{\sigma_{a}\colon a\in R\}$, where $\sigma_{a}$ is a
reflection with respect to hyperplane $(a,x)=0$. Let $k(a)\colon R\to
\mathbb{R}_{+}$ be a $G$-invariant multiplicity function.
%Recall that a finite
%subset $R\subset \mathbb{R}^{d}\setminus\{0\}$ is called a root system, if
%$R\cap\mathbb{R} a=\{\pm a\}$~and $\sigma_{a}R=R$ for all $a\in R$.

Let
\[
v_{k}(x)=\prod_{a\in R_{+}}|(a,x)|^{2k(a)}
\]
be the Dunkl weight,
\[
d\mu_{k}(x)=c_{k}v_{k}(x)\,dx,\quad
c_{k}^{-1}=\int_{\mathbb{R}^{d}}e^{-|x|^{2}/2}v_{k}(x)\,dx,
\]
and $L^{p}(\mathbb{R}^{d},d\mu_{k})$, $1\le p<\infty$, be the space of
complex-valued Lebesgue measurable functions $f$ for which
\[
\|f\|_{p}=\|f\|_{p,d\mu_{k}}=\biggl(\int_{\mathbb{R}^{d}}|f|^{p}\,d\mu_{k}\biggr)^{1/p}<\infty.
\]
We also assume that $L^{\infty}\equiv C_b$ is the space of bounded continuous
functions $f$ with norm $\|f\|_{\infty}$. As usual $\mathcal{S}(\mathbb{R}^d)$
denotes the Schwartz space.

The differential-differences Dunkl operators are given by
\[
D_{j,k}f(x)=\frac{\partial f(x)}{\partial x_{j}}+ \sum_{a\in
R_{+}}k(a)(a,e_{j})\,\frac{f(x)-f(\sigma_{a}x)}{(a,x)},\quad j=1,\dots,d.
\]
Let $\Delta_k=\sum_{j=1}^dD_{j,k}^2$ be the Dunkl Laplacian. As usual,
$(-\Delta_k)^{s}$ for $s>0$ stands for the fractional power of the Dunkl
Laplacian; see Subsection~\ref{subsec-lapl} for more details.

By definition,
\begin{equation}\label{Weighteddimension}
\lambda_k=\frac{d}{2}-1+\sum_{a\in R_+}k(a)\quad \text{and}\quad
d_k=2(\lambda_k+1).
\end{equation}
The number $d_k$ plays the role of the generalized dimension of the space
$(\mathbb{R}^d,d\mu_k)$. We note that $\lambda_k\geq-1/2$ and, moreover,
$\lambda_k=-1/2$ if and only if $d=1$ and $k\equiv0$. In what follows we assume
that
\[
\lambda_k>-\frac{1}{2}\quad \text{and}\quad d_k>1.
\]

%Let $E_{\sigma}(f)_{p}$ be the value of the best approximation of a function
%$f$ in the space $L^{p}(\mathbb{R}^{d},d\mu_{k})$ by entire functions of
%spherical exponential type at most $\sigma>0$.

Throughout this paper, $A\lesssim B$ means that $A\leq CB$ with a constant
$C>0$ depending only on nonessential parameters. Moreover, we write $A\asymp B$
if $A\lesssim B$ and $B\lesssim A$.
%\textbf{Через $\chi_{E}$ обозначается характеристическая функция множества $E$}.

\subsection{Sharp Jackson and inverse inequalities}

Let $E_{\sigma}(f)_{p}$ be the best approximation of a function $f$ in
$L^{p}(\mathbb{R}^{d},d\mu_{k})$ by entire functions of type $\sigma>0$, i.e.,
\[
E_{\sigma}(f)_{p}=\inf_{g\in \mathcal{B}_{p,k}^{\sigma}}\|f-g\|_{p},
\]
{where
$\mathcal{B}_{p,k}^{\sigma}=\mathcal{B}_{p,k}^{\sigma}(\mathbb{R}^{d})$ is the
Bernstein class of entire functions of spherical exponential type at most
$\sigma$ from $L^p(\mathbb{R}^d,d\mu_k)$} (see
Subsection~\ref{subsec-bern} for more details). It is known \cite{GorIvaTik19}
that the best approximation is achieved. As usual, $\omega_r(f,\delta)_{p}$
denotes the modulus of smoothness of order $r$ of $f\in
L^{p}(\mathbb{R}^{d},d\mu_{k})$ (the precise definition and the main properties
are given in Subsection~\ref{subsec-mod}).

In \cite{GorIvaTik19,GorIva19}, we derived the classical Jackson and inverse
approximation theorems in $L^{p}(\mathbb{R}^{d},d\mu_{k})$, $1\le p\le\infty$,
namely,
\begin{equation}\label{eq1}
E_{\sigma}(f)_{p}\lesssim \omega_r\Bigl(f,\frac{1}{\sigma}\Bigr)_{p},\quad \sigma, r>0,
\end{equation}
and
\begin{equation}\label{eq2}
\omega_{r}\Bigl(f,\frac{1}{n}\Bigl)_{p}\lesssim \frac{1}{n^{r}} \sum_{j=0}^n (j+1)^{r-1} E_j(f)_{p},\quad n\in\mathbb{N},
\quad r>0,
\end{equation}
as well as the equivalence between the fractional modulus of smoothness and the $K$-functional:
\begin{equation}\label{eq3}
K_{r}(f,\delta)_{p}\asymp \omega_r(f,\delta)_{p}, \quad \delta>0.
\end{equation}
Moreover, we obtained that for $d_{k}>1$
\begin{equation}\label{omega-delta}
\omega_r(f,\delta)_{p}=\sup_{0<t\leq\delta}\|\varDelta_t^rf\|_{p}\asymp
\|\varDelta_\delta^rf\|_{p},\quad \delta>0,
\end{equation}
where the difference $\varDelta_\delta^r$ is defined by the generalized
translation operator $T^t$ (see \eqref{difference} below). In \cite{GorIva19b},
we proved the Jackson inequality in $L^{p}(\mathbb{R}^{d},d\mu_{k})$, $1\le
p<2$, with a sharp constant.

Our first goal is to sharpen \eqref{eq1} and \eqref{eq2} in the case
$1<p<\infty$, taking into account the strict convexity of the spaces
$L^{p}(\mathbb{R}^{d},d\mu_{k})$. The sharp Jackson and sharp inverse
inequalities are given in the following result.

\begin{theorem}\label{thm1}
If $1<p<\infty$, $r>0$, $n\in\mathbb{N}$, $s=\max(p,2)$, and $q=\min(p,2)$, then for any $f\in
L^{p}(\mathbb{R}^{d},d\mu_{k})$,
\begin{equation}\label{eq4}
\frac{1}{n^r}\biggl(\sum_{j=1}^{n}j^{sr-1}E_{j}^s(f)_{p}\biggr)^{1/s}\lesssim
\omega_r\Bigl(f,\frac{1}{n}\Bigr)_{p}
\end{equation}
and
\begin{equation}\label{eq6}
\omega_{r}\Bigl(f,\frac{1}{n}\Bigl)_{p}\lesssim
\frac{1}{n^{r}}\biggl(\sum_{j=1}^n j^{qr-1}
E_j^q(f)_{p}\biggr)^{1/q}+\frac{\|f\|_{p}}{n^r}.
\end{equation}
\end{theorem}
Inequalities \eqref{eq4} and \eqref{eq6} have been first obtained by M.F. Timan
(see \cite{devore,Tim58,Tim66}) for periodic functions $f\in L^p(\mathbb{T})$:
\[
\frac{1}{n^r}\biggl(\sum_{j=1}^{n}j^{sr-1}E_{j-1}^s(f)_{L^p(\mathbb{T})}\biggr)^{1/s}\lesssim
\omega_r\Bigl(f,\frac{1}{n}\Bigr)_{L^p(\mathbb{T})}\lesssim
\frac{1}{n^{r}}\biggl(\sum_{j=1}^n j^{qr-1}E_{j-1}^q(f)_{L^p(\mathbb{T})}\biggr)^{1/q},
\]
where $\mathbb{T}=(-\pi,\pi]$, $r\in\mathbb{N}$, $E_{j}(f)_{L^p(\mathbb{T})}$ is the
best approximation of $f$ by trigonometric polynomials of degree at most $j$,
and $\omega_r(f,\delta)_{L^p(\mathbb{T})}$ is the classical $r$-modulus of smoothness.
Sharp Jackson and inverse approximation inequalities were further developed in
many papers (see for example
\cite{DaiDit05,DaiDit07,DaiDitTik08,Dit88,DitPry07, kol,Tot88} and the
references therein).
%If the Timan proofs were based on the boundedness of the partial sums
%of the Fourier series in $L^p(-\pi,\pi]$, $1<p<\infty$, then the authors of
%these papers used the geometry of $L_p$-spaces and the Littlewood--Paley-type
%inequalities.
Our proof of Theorem \ref{thm1} is based on the corresponding Littlewood--Paley decomposition in the
Dunkl setting; cf.~\cite{DaiDit05,DaiDitTik08}.

The next two inequalities provide sharp interrelation between fractional moduli
of smoothness of different orders.

\begin{corollary}\label{thm2}
Under the assumptions of Theorem \ref{thm1}, the following sharp reverse Marchaud
and sharp Marchaud inequalities hold: for any $m>r>0$,
%If $1<p<\infty$, $m>r>0$, $n\in\mathbb{N}$, $q=\min(p,2)$, then for any $f\in L^{p}(\mathbb{R}^{d},d\mu_{k})$
\begin{equation}\label{eq5}
\frac{1}{n^r}\biggl(\sum_{j=1}^{n}j^{sr-1}\omega_m^s\Bigl(f,\frac{1}{j}\Bigr)_{p}\biggr)^{1/s}\lesssim
\omega_r\Bigl(f,\frac{1}{n}\Bigr)_{p}+\frac{\|f\|_{p}}{n^r}
\end{equation}
and \begin{equation}\label{eq7}
\omega_{r}\Bigl(f,\frac{1}{n}\Bigl)_{p}\lesssim
\frac{1}{n^{r}}\biggl(\sum_{j=1}^n j^{qr-1}
\omega_{m}^q\Bigl(f,\frac{1}{j}\Bigr)_{p}\biggr)^{1/q}+\frac{\|f\|_{p}}{n^r}.
\end{equation}
\end{corollary}

\subsection{Smoothness of functions via smoothness of best approximants}
%Recently, Yu.S. Kolomoitsev and S.Yu. Tikhonov
Recently \cite{KolTik19}, %it has proposed to use
 smoothness properties of approximation processes were used to characterize smoothness properties of
functions themselves.
We continue this line of research in $L^p$ with Dunkl weights.
As approximation processes we consider the best approximants and the de la
Vall\'{e}e Poussin type operators.

%Recall that $\eta\in \mathcal{S}_{\mathrm{rad}}(\mathbb{R}^d)$ such that
%$\eta(x)=1$ if $|x|\leq 1/2$, $\eta(x)>0$ if $|x|<1$, and $\eta(x)=0$ if
%$|x|\geq 1$.
%Set
%$\eta_j(x)=\eta(2^{-j}x)
%$
%and let
%$\eta_jf$ be the multiplier linear operator given by
%$\mathcal{F}_{k}(\eta_jf)(y)=\eta_j(y)\mathcal{F}_{k}(f)(y).$

Let $f_{\sigma}\in \mathcal{B}_{p,k}^{\sigma}$ denote the best approximant of
$f$ in $L^{p}(\mathbb{R}^{d},d\mu_{k})$, that is,
$E_{\sigma}(f)_{p}=\|f-f_{\sigma}\|_{p}$. Assume that $\eta_{j}f$ is the de la
Vall\'{e}e Poussin type operator, namely, $\eta_jf$ is the multiplier linear
operator given by
$%\[
\mathcal{F}_{k}(\eta_jf)(y)=\eta_j(y)\mathcal{F}_{k}(f)(y).
%,\quad \mathcal{F}_{k}(\theta_jf)(y)=\theta_j(y)\mathcal{F}_{k}(f)(y),
$
Here $\eta_j(x)=\eta(2^{-j}x)$ and
%defined by
a radial function $\eta\in \mathcal{S}(\mathbb{R}^d)$ is such that $\eta(x)=1$
if $|x|\leq 1/2$, $\eta(x)>0$ if $|x|<1$, and $\eta(x)=0$ if $|x|\geq 1$; see
Section~\ref{sec4} for more details.

%, let $W^{r}_{p,k}$, $r>0$, be the
%Sobolev space, let $K_{r}(f,t)_{p}$ be the $K$-functional for the couple
%$(L^{p}(\mathbb{R}^{d},d\mu_{k}), W^{r}_{p,k})$.
% Denote by $\mathcal{F}_{k}$,
%$\mathcal{F}=\mathcal{F}_{0}$, and $\mathcal{H}_{\lambda}$ and the Dunkl, the
%Hankel and the Fourier transforms respectively. Exact definitions will be given
%below.

\begin{theorem}\label{thm3}
If $1<p<\infty$, $r>0$, $n\in\mathbb{N}$, $s=\max(p,2)$, and $q=\min(p,2)$,
then for any $f\in L^{p}(\mathbb{R}^{d},d\mu_{k})$,
\begin{align}
\biggl(\sum_{j=n+1}^{\infty}2^{-srj}\bigl\|(-\Delta_k)^{r/2}P_{j}f\bigr\|_{p}^s\biggr)^{1/s}&\lesssim
\omega_r(f,2^{-n})_{p}\notag\\&\lesssim
\biggl(\sum_{j=n+1}^{\infty}2^{-qrj}\bigl\|(-\Delta_k)^{r/2}P_{j}f\bigr\|_{p}^q\biggr)^{1/q},\label{eq8}
\end{align}
where
$P_{j}f$ stands for the best approximants $f_{2^{j}}$ or the de la Vall\'{e}e Poussin
type operators~$\eta_{j}f$.
\end{theorem}

%\begin{theorem}\label{thm4}
%Under the assumptions of Theorem \ref{thm3}, we have
%\begin{align}
%\biggl(\sum_{j=n+1}^{\infty}2^{-srj}\bigl\|(-\Delta_k)^{r/2}\eta_{j}f\bigr\|_{p}^s\biggr)^{1/s}&\lesssim
%\omega_r(f,2^{-n})_{p}\notag\\&\lesssim
%\biggl(\sum_{j=n+1}^{\infty}2^{-qrj}\bigl\|(-\Delta_k)^{r/2}\eta_{j}f\bigr\|_{p}^q\biggr)^{1/q}.
%\label{eq9}
%\end{align}
%\end{theorem}

\subsection{Weighted Fourier inequalities in Dunkl setting}
In various problems of harmonic analysis and approximation theory it is
important to know how smoothness of functions is related to the behavior of its
Fourier transforms. This study was originated by Lebesgue \cite[(4.1)]{Zy} who
obtained the following estimate for the Fourier coefficients~$\widehat{f}_n$ of
a periodic function $f\in L^{1}(\mathbb{T})$:
\begin{equation}\label{lebesgue}
|\widehat{f}_n|\lesssim \omega_1\Bigl(f,\frac1n\Bigr)_{L^{1}(\mathbb{T})}\quad\mbox{as}\quad
n\to \infty.
\end{equation}
%where $\omega_1(f,\frac1n)_1$ is the $L^1$-modulus of continuity of $f$.
Similar problems for the Fourier transform/coefficients in
${L^{p}(\mathbb{R}^d)}$ and ${L^{p}(\mathbb{T}^d)}$
 %the non-weighted Lebesgue space, i.e., $k\equiv 0$,
  have
been recently investigated in \cite{BraPin08,Bra13,GorTik12}. In this paper we not only
extend these results for the Dunkl setting but obtain completely new Fourier inequalities. Let $\mathcal{F}_{k}(f)$ denote the
Dunkl transform, see Section~\ref{sec2}. {For $k\equiv 0$ we deal with  the usual Fourier transform $\mathcal{F}_0(f)=\widehat{f}$.}
%\footnote{S: teper kak-to ne logichno poluchaetsja - snachala byli ocenki dlya
%modulej, potom neravenstva typa Pitta. Teper kasha. Nuzhna teorema pro moduli
%dlya Kellogga}
Let $\chi_{j}$ be the
 characteristic functions of the dyadic annuli $\{2^{j}\le |x|<2^{j+1}\}$, $j\in
\mathbb{Z}$, that is,
$\chi_{j}=\chi_{\{2^{j}\le |x|<2^{j+1}\}}$.

We obtain the following estimates of moduli of smoothness in terms of Dunkl transforms.

\begin{theorem}\label{thm5}
Let $\delta,r>0$.

\textup{(1)} If $1<p\le 2$ and $f\in L^p(\mathbb{R}^d,
d\mu_k)$, then, for  $p\le q\le p'$,
\[
\bigl\||x|^{d_{k}(1/p'-1/q)}\min\{1,(\delta|x|)^{r}\}\mathcal{F}_{k}(f)\bigr\|_{q}\lesssim
\omega_r(f,\delta)_{p}
\]
and
\[
\biggl(\sum_{j\in
\mathbb{Z}}\min\{1,(2^{j}\delta)^{2r}\}\bigl\|\mathcal{F}_{k}(f)\chi_{j}\bigr\|_{p'}^{2}\biggr)^{1/2}\lesssim
\omega_r(f,\delta)_{p}.
\]

\textup{(2)} If $2\le p<\infty$, $p'\le q\le p$, and $f\in
\mathcal{S}'(\mathbb{R}^d)$ is such that
$|\Cdot|^{d_{k}(1/p'-1/q)}\mathcal{F}_{k}(f)\in L^q(\mathbb{R}^d,d\mu_k)$,
then $f\in L^p(\mathbb{R}^d, d\mu_k)$ and
\begin{equation}\label{omega-pitt}
\omega_r(f,\delta)_{p}\lesssim
\bigl\||x|^{d_{k}(1/p'-1/q)}\min\{1,(\delta|x|)^{r}\}\mathcal{F}_{k}(f)\bigr\|_{q}.
\end{equation}

%\textup{(2)} If $1<p\le 2$ and $f\in L^p(\mathbb{R}^d, d\mu_k)$, then
If $2\le p<\infty$ and  $f\in \mathcal{S}'(\mathbb{R}^d)$ is such that
$\bigl(\sum_{j\in
\mathbb{Z}}\,\bigl\|\mathcal{F}_{k}(f)\chi_{j}\bigr\|_{p'}^{2}\bigr)^{1/2}<\infty$,
then $f\in L^p(\mathbb{R}^d, d\mu_k)$ and
\[
\omega_r(f,\delta)_{p}\lesssim \biggl(\sum_{j\in
\mathbb{Z}}\min\{1,(2^{j}\delta)^{2r}\}\bigl\|\mathcal{F}_{k}(f)\chi_{j}\bigr\|_{p'}^{2}\biggr)^{1/2}.
\]
\end{theorem}

\begin{remark}
(i) An analogue of Lebesgue-type estimate \eqref{lebesgue} for the Dunkl
transform is given as follows. If $f\in L^1(\mathbb{R}^d,
d\mu_k)$, then we simply have %, for  $p\le q\le p'$,
\[
\bigl|\mathcal{F}_{k}(f)(x)\bigr|\lesssim
\omega_r\Big(f,\frac{1}{|x|}\Big)_{1}.
\]
This estimate can be  equivalently written as
\[
\bigl\|\min\{1,(\delta|x|)^{r}\}\mathcal{F}_{k}(f)\bigr\|_{\infty}\lesssim
\omega_r(f,\delta)_{1};
\]
see \eqref{omega-lambda} below.

(ii) In Theorem \ref{thm5} one can replace
$\omega_r(f,\delta)_{p}$  the difference  $\|\varDelta_t^rf\|_{p}$;
cf.  \eqref{omega-delta}.
% the modulus of
%smoothness in the definitions of the Besov--Dunkl space can be equivalently
%replaced by.
\end{remark}
{To prove this theorem, we need the following Pitt- {and Kellogg-}type
inequalities, which are of interest by themselves.}

\begin{theorem}\label{thm-pitt}
\textup{(1)} If $1<p\le 2$ and $f\in
L^p(\mathbb{R}^d,d\mu_k)$, then for $p\le q\le p'$,
\begin{equation}\label{pitt1}
\bigl\||x|^{d_{k}(1/p'-1/q)}\mathcal{F}_{k}(f)\bigr\|_{q}\lesssim
\|f\|_{p}
\end{equation}
and
\begin{equation}\label{kell1}
\biggl(\sum_{j\in
\mathbb{Z}}\,\bigl\|\mathcal{F}_{k}(f)\chi_{j}\bigr\|_{p'}^{2}\biggr)^{1/2}\lesssim
\|f\|_{p}.
\end{equation}
\textup{(2)} If $2\le p<\infty$, $p'\le q\le p$, and $f\in
\mathcal{S}'(\mathbb{R}^d)$ is such that
$|\Cdot|^{d_{k}(1/p'-1/q)}\mathcal{F}_{k}(f)\in L^q(\mathbb{R}^d,d\mu_k)$, then
\begin{equation}\label{pitt2}
\|f\|_{p}\lesssim \bigl\||x|^{d_{k}(1/p'-1/q)}\mathcal{F}_{k}(f)\bigr\|_{q}.
\end{equation}
% If $1<p\le 2$ and $f\in L^p(\mathbb{R}^d,d\mu_k)$, then
If $2\le p<\infty$ and $f\in \mathcal{S}'(\mathbb{R}^d)$ is such that
$\bigl(\sum_{j\in
\mathbb{Z}}\,\bigl\|\mathcal{F}_{k}(f)\chi_{j}\bigr\|_{p'}^{2}\bigr)^{1/2}<\infty$,
then
\begin{equation}\label{kell2}
\|f\|_{p}\lesssim \biggl(\sum_{j\in
\mathbb{Z}}\,\bigl\|\mathcal{F}_{k}(f)\chi_{j}\bigr\|_{p'}^{2}\biggr)^{1/2}.
\end{equation}
\end{theorem}

\begin{remark}%\label{rem-K}
%Мы объединили классические неравенства Питта (1) и неравенства Келлогга (2) в
%одно утверждение, поскольку интересно их сравнить.
(i) It is worth mentioning that for  $2\le q\le p'$ Kellogg-type inequality \eqref{kell1} improves Pitt's inequality
\eqref{pitt1} since
\begin{equation}\label{K-1}
\bigl\||x|^{d_{k}(1/p'-1/q)}\mathcal{F}_{k}(f)\bigr\|_{q}\lesssim \biggl(\sum_{j\in
\mathbb{Z}}\,\bigl\|\mathcal{F}_{k}(f)\chi_{j}\bigr\|_{p'}^{2}\biggr)^{1/2}.
\end{equation}
Similarly, if $p'\le q\le 2$, inequality \eqref{kell2}) sharpens   Pitt's
inequality \eqref{pitt2} since in this case one has
\begin{equation}\label{K-2}
\biggl(\sum_{j\in
\mathbb{Z}}\,\bigl\|\mathcal{F}_{k}(f)\chi_{j}\bigr\|_{p'}^{2}\biggr)^{1/2}\lesssim
\bigl\||x|^{d_{k}(1/p'-1/q)}\mathcal{F}_{k}(f)\bigr\|_{q}.
\end{equation}
It is easy to construct  examples of  functions showing that the behavior of the left-hand and right-hand sides in % the sharpness of inequalities
 \eqref{K-1} and \eqref{K-2} is different; see Remark~\ref{rem-pitt-kell} below.
\end{remark}

(ii)
Pitt's inequalities are well known in the non-weighted case ($k\equiv 0$); see,
e.g., \cite{BH03,GLT11,ind}. Taking $q=p'$ and $q=p$, \eqref{pitt1} and
\eqref{pitt2} become analogues of the Hausdorff--Young and Hardy--Littlewood
inequalities for Dunkl transform; see \cite{AASS09}. Below we give a simple proof of
Theorem~\ref{thm-pitt} based on the interpolation technique \cite{Ste56} and
on the Hardy--Littlewood inequality
\cite{AASS09}.
See also \cite{AM15} for some extensions of the Hardy--Littlewood inequality.
%The proof of Theorem \ref{thm-pitt} is based
%on the Hardy--Littlewood inequality for the Dunkl transform given in
%\cite{AASS09}.

For trigonometric series  $f(x)\sim \sum_{n\in \mathbb{Z}\setminus
\{0\}}\widehat{f}_{n}\,e^{inx}$ Kellogg's inequality  \cite{Ke71} states  that for
 $1<p\le 2$
\[%\begin{equation}\label{K-ineq}
\biggl(\sum_{n\in \mathbb{Z}\setminus
\{0\}}|\widehat{f}_{n}|^{p'}\biggr)^{1/p'}\le
\biggl(\sum_{j=0}^{\infty}\Big(\sum_{2^{j}\le
|n|<2^{j+1}}|\widehat{f}_{n}|^{p'}\Big)^{2/p'}\biggr)^{1/2}\lesssim
\|f\|_{L^{p}(\mathbb{T})},
\]%\end{equation}
improving the Hausdorff--Young inequality. The reverse estimates are valid for
$2\le p<\infty$. The  example $\sum_{l=1}^{N}l^{-1/2}\cos 2^{l}x$ shows the
advantages to work with Kellogg's inequality rather than with
Hausdorff--Young's inequality.
% optimality показывает точность первого неравенства в \eqref{K-ineq}.
For Fourier transforms on $\mathbb{R}^{d}$ Kellogg-type estimate was obtained by Kurtz \cite{Ku80}.
%Для преобразования Фурье и диадических разбиений $\mathbb{R}^{d}$ по
%параллелепипедам аналоги неравенства Kellogg установлены Курцем \cite{Ku80}.

%Для удобства введем весовую диадическую
%норму
%\[
%\|g\|_{K_{p,q}}= \biggl(\sum_{j\in
%\mathbb{Z}}\,\bigl\||x|^{d_{k}(1/p'-1/q)}g\chi_{j}\bigr\|_{q}^{2}\biggr)^{1/2},\quad
%1<p,q<\infty.
%\]

%\textbf{!!!здесь и везде, где произведение типа $\chi_{j}\cdot g$ без
%аргументов сделал $g\chi_{j}$, иначе было ужасно (а писать аргументы громоздко)!!!}
%\footnote{S: nizhe izmenil}
%\begin{theorem}\label{thm-K}
%\textup{(1)} Для $1<p\le 2$ %, $p\le q\le p'$
%\begin{equation}\label{K-1-}
%%\|\mathcal{F}_{k}(f)\|_{K_{p,q}}\lesssim
%%\|\mathcal{F}_{k}(f)\|_{K_{p,p'}}
%\biggl(\sum_{j\in
%\mathbb{Z}}\,\bigl\|\mathcal{F}_{k}(f)\chi_{j}\bigr\|_{p'}^{2}\biggr)^{1/2}
%\lesssim \|f\|_{p}.
%\end{equation}
%
%\textup{(2)} Для $2\le p<\infty$ %, $p'\le q\le p$
%\begin{equation}\label{K-2-}
%\|f\|_{p}\lesssim \biggl(\sum_{j\in
%\mathbb{Z}}\,\bigl\|\mathcal{F}_{k}(f)\chi_{j}\bigr\|_{p'}^{2}\biggr)^{1/2}.%\|f\|_{K_{p,p'}}\lesssim \|f\|_{K_{p,q}},
%\end{equation}
%\end{theorem}

%В доказательстве этой теоремы для мы приведем пример функции $f_{N}$, $N\gg 1$,
%такой что для $2<q\le p'$ имеем $\|\mathcal{F}_{k}(f_{N})\|_{K_{p,q}}\asymp
%(\ln N)^{1/2}$ и
%$\bigl\||x|^{d_{k}(1/p'-1/q)}\mathcal{F}_{k}(f_{N})\bigr\|_{q}\asymp 1$. Этот
%пример показывает неэквивалентность неравенств \eqref{K-1}, отвечающих
%неравенствам Келлогга и Питта. Аналогичная ситуация в случае обратного
%неравенства \eqref{K-2}.

\subsection{Characterizations of the Besov spaces}
It is well known that the classical Besov spaces on $\mathbb{R}^d$ can be
equivalently defined Fourier analytically or  in terms of differences (moduli
of smoothness); see, e.g., \cite[Ch.~3.5]{Tri92}.
%  characterized in terms of best approximations.
 Another characterization of Besov spaces via smoothness of
approximation processes has  been suggested in \cite{KolTik19}.

A detailed study of the  Besov--Dunkl space has started in the last decade. To
define it, the authors  usually use  Fourier-analytical decompositions
\[
\|f\|_{\dot{B}_{p,\vartheta}^{s}}=
\biggl(\sum_{j=-\infty}^{\infty}2^{s\vartheta
j}\|\theta_{j}f\|_{p}^{\vartheta}\biggr)^{1/\vartheta},\quad
\theta_{j}=\eta_{j}-\eta_{j-1},
\]
%где $\theta_{j}=\eta_{j}-\eta_{j-1}$.
(we refer to \cite{AASS09}  and
the references therein). Here we would like to obtain various characterizations
of the  Besov--Dunkl space.
%were considered in several papers (we refer to \cite{AASS09} and the references therein).
Let us introduce the
(inhomogeneous) Besov--Dunkl space in terms of moduli of smoothness. %; cf. \cite[Ch.~2]{Sa18}, \cite[Ch.~2]{Tri92}.}

Let $1<p<\infty$, $0<\vartheta\le \infty$, and $s>0$. We say that $f\in
L^p(\mathbb{R}^d,d\mu_k)$ belongs to the Besov--Dunkl space
$B_{p,\vartheta}^{s}=B_{p,\vartheta}^{s}(\mathbb{R}^{d},d\mu_{k})$ if
\[
\|f\|_{B_{p,\vartheta}^{s}}= \|f\|_p+
\biggl(\int_{0}^{1}\bigl(t^{-s}\omega_{r}(f,t)_{p}\bigr)^{\vartheta}\,
\frac{dt}{t}\biggr)^{1/\vartheta}<\infty,\quad \vartheta<\infty,\ s<r,
\]
and
\[
\|f\|_{B_{p,\infty}^{s}}=
\|f\|_p+
\sup_{t>0} \frac{\omega_{r}(f,t)_{p}}{t^{s}}<\infty,\quad
\vartheta=\infty,\ s\le r.
\]
Sometimes the space $B_{p,\infty}^{s}$ is called the Lipschitz space.

\begin{remark}
It is important to mention that in light of  \eqref{omega-delta} the modulus of
smoothness in the definitions of the Besov--Dunkl space can be equivalently
replaced by the difference  $\|\varDelta_t^rf\|_{p}$. This sometimes is more
frequently used to define the Besov norm in the classical case ($k\equiv 0$). For the one-dimensional Besov--Dunkl space,  see, e.g.,
%Для одномерного преобразования Данкля на $\mathbb{R}$ подобные классы Бесова рассмотрены, например, в работе
 \cite{Ka14}.
 See also \cite{KM12} for more information on  inhomogeneous Besov--Dunkl spaces % Мы отсылаем к работе \cite{KM12} за дальнейшими результатами для %неоднородных классов Бесова--Данкля
 %(that is, given by  \eqref{B-theta})
  and their embeddings.
 %, включающими теоремы вложения.

\end{remark}

%\textbf{Для сравнения в работе \cite{AASS09} рассмотрен случай homogeneous
%Besov--Dunkl space с нормой (в наших обозначениях)}
%\[
%\|f\|_{\dot{B}_{p,\vartheta}^{s}}= \biggl(\sum_{j=-\infty}^{\infty}2^{s\vartheta
%j}\|\theta_{j}f\|_{p}^{\vartheta}\biggr)^{1/\vartheta},
%\]
%где $\theta_{j}=\eta_{j}-\eta_{j-1}$. Одной из целей работы \cite{AASS09}
%является доказательство аналогов Szasz' theorem.

%Equivalence~\eqref{omega-delta}, Theorems \ref{thm1}, \ref{thm3} and
%\ref{thm4}, Corollary \ref{thm2}, and Hardy's inequalities immediately imply
%the following

%Приведем следующие эквивалентные определения нормы в неоднородном классе
%Бесов--Данкла.

\begin{theorem}\label{thm-besov}
\textup{(1)} %{\bf Определение пространства Бесова--Данкля не зависит от выбора $r>s$.} (
The (quasi-)norms $\|f\|_{B_{p,\vartheta}^{s}}$ do not depend on
the choice of $r>s$.
\\
\textup{(2)} %For any $1<p<\infty$ and $0<\vartheta\le \infty$
The following characterizations hold:
\begin{align}
\|f\|_{B_{p,\vartheta}^{s}}&\asymp
\|f\|_p+ \biggl(\sum_{j=0}^{\infty}2^{s\vartheta j}
\bigl(\omega_{r}(f,2^{-j})_{p}\bigr)^{\vartheta}\biggr)^{1/\vartheta}\label{B-omega}\\
&\asymp \|f\|_p+
\biggl(\sum_{j=0}^{\infty}2^{s\vartheta j}
\bigl(E_{2^j}(f)_p\bigr)^\vartheta\biggr)^{1/\vartheta}\label{B-E}\\
&\asymp \|f\|_p+ \biggl(\sum_{j=1}^{\infty}2^{s\vartheta j}
\|f-\eta_{j}f\|_{p}^\vartheta\biggr)^{1/\vartheta}\label{B-eta}\\
&\asymp \|f\|_p+ \biggl(\sum_{j=1}^{\infty}2^{s\vartheta j}
\|\theta_{j}f\|_{p}^\vartheta\biggr)^{1/\vartheta}\label{B-theta}\\
&\asymp \|f\|_p+
\biggl(\sum_{j=-\infty}^{\infty}2^{s\vartheta j}
\|\theta_{j}f\|_{p}^\vartheta\biggr)^{1/\vartheta}
%=\|f\|_p+\|f\|_{\dot{B}_{p,\infty}^{s}}
\label{B-B}\\
&\asymp
\|f\|_p+ \biggl(\sum_{j=1}^{\infty}2^{(s-r)\vartheta j}
\bigl\|(-\Delta_{k})^{r/2}P_{j}f\bigr\|_{p}^\vartheta\biggr)^{1/\vartheta},\label{B-P}
\end{align}
where $P_{j}f$ stands for the best approximants $f_{2^{j}}$ or the de la
Vall\'{e}e Poussin type operators~$\eta_{j}f$.
\end{theorem}

Let us also give necessary (for $1<p\le 2$) and sufficient  (for $2\le p<\infty$) conditions for $f$ to belong to the
Besov--Dunkl space given in terms of behavior of its Fourier--Dunkl transform.

\begin{theorem}\label{thm-besov-pitt}
\textup{(1)}
If $1<p\le 2$ and $f\in B_{p,\vartheta}^{s}$, then
\[
\bigl\|\mathcal{F}_{k}(f)\bigr\|_{p'}+ \biggl(\sum_{j=0}^{\infty}2^{s\vartheta
j} \bigl\|\mathcal{F}_{k}(f)\chi_{j}\bigr\|_{p'}^{\vartheta}\biggr)^{1/\vartheta}\lesssim
\|f\|_{B_{p,\vartheta}^{s}}.
\]
\textup{(2)}  If $2\le p<\infty$ and $f\in \mathcal{S}'(\mathbb{R}^d)$ is such
%that  $\mathcal{F}_{k}(f)\in L^{p'}(\mathbb{R}^d,d\mu_k)$
that  $\mathcal{F}_{k}(f)\in L^{p'}(\mathbb{R}^d,d\mu_k)$ and
$\bigl(\sum_{j=0}^{\infty}2^{s\vartheta
j} \bigl\|\mathcal{F}_{k}(f)\chi_{j}\bigr\|_{p'}^{\vartheta}\bigr)^{1/\vartheta}<\infty$, then
\[
\|f\|_{B_{p,\vartheta}^{s}}\lesssim
\bigl\|\mathcal{F}_{k}(f)\bigr\|_{p'}+ \biggl(\sum_{j=0}^{\infty}2^{s\vartheta
j} \bigl\|\mathcal{F}_{k}(f)\chi_{j}\bigr\|_{p'}^{\vartheta}\biggr)^{1/\vartheta}.
\]
%Then under the conditions of \ref{thm5}~\textup{(1)} we have $S(f)\lesssim
%\|f\|_{B_{p,\vartheta}^{s}},$ and under the conditions of
%\ref{thm5}~\textup{(2)} we have
%а в условиях~(2) соответственно справедливо обратное неравенство
%$\|f\|_{B_{p,\vartheta}^{s}}\lesssim S(f)$.
%\footnote{S: mne kazhetsja, nado zapisat' takuju zhe teoremu dlya Kellogga, inache srazu voznikaut voprosy.}
\end{theorem}

As a simple application of Theorem \ref{thm-besov-pitt} we establish
the following characterization of the Besov--Dunkl space for $p=2$. % $p=q=2$ получаем следующее

\begin{corollary}
For $f\in L^2(\mathbb{R}^d,d\mu_k)$ we have
\[
\|f\|_{B_{2,\vartheta}^{s}}\asymp \bigl\|\mathcal{F}_{k}(f)\bigr\|_2+
\biggl(\sum_{j=0}^{\infty}2^{s\vartheta
j} \bigl\|\mathcal{F}_{k}(f)\chi_{j}\bigr\|_{2}^{\vartheta}\biggr)^{1/\vartheta}.
\]
\end{corollary}
Moreover, taking $\vartheta=\infty$,
%is called the  space. ?????В случае класса Липшица $\vartheta=\infty$
we arrive at a Titchmarsh type result for the Lipschitz space
$B_{2,\infty}^{s}$:
\begin{equation}\label{tit}
\|f\|_{B_{2,\infty}^{s}}\asymp \bigl\|\mathcal{F}_{k}(f)\bigr\|_2+ \sup_{j\in
\mathbb{Z}_{+}}2^{sj} \bigl\|\mathcal{F}_{k}(f)\chi_{j}\bigr\|_{2},
\end{equation}
extending the main result of \cite{Ma10}.
%\footnote{S: eto zhe pravda? tam takoe zhe opredelenie kak i u nas?}
Recall that the classical Titchmarsh theorem
\cite[Theorem~85]{Ti48} states that for $0<\alpha<1$ the condition
\[
\|f({\cdot}+h)-f({\cdot})\|_{L^2(\mathbb{R})}=O(h^\alpha)\quad \text{as}\quad
h\to 0
\]
is equivalent to the condition
\[
\biggl(\int_{t<|\xi|}|\widehat{f}\,(\xi)|^2\,d\xi\biggr)^{1/2}=O(t^{-\alpha})\quad
\text{as}\quad t\to \infty.
\]
The latter can be equivalently written as
%\[
%\biggl(\int_{2^j\le
%|\xi|<2^{j+1}}|\widehat{f}\,(\xi)|^2\,d\xi\biggr)^{1/2}=O(2^{-j\alpha})\quad
%\text{as}\quad j\to \infty
%\]
%or, ,
$%\[
\sup_{j\in \mathbb{Z}_{+}}2^{j\alpha} \|\widehat{f}\,\chi_{j}\|_{2}<\infty;
$ %\]
cf. the right-hand side of \eqref{tit}.

\subsection{Structure of the paper}
The paper is organized as follows.
In the next section, we introduce some basic notation and present important auxiliary results of the Dunkl harmonic analysis. %including Young's and Hardy--Littlewood inequalities.
In Section~\ref{sec3}, we introduce needed spaces of distributions. Moreover,
we define the fractional power of the Dunkl Laplacian, the fractional modulus
of smoothness, and the fractional $K$-functional associated to the Dunkl
weight.
%, as well as the best approximation by entire functions of spherical exponential type.

Section~\ref{sec4} contains the Littlewood--Paley-type inequalities in the
Dunkl setting. In Section~\ref{sec5}, we prove the sharp direct and inverse
theorems of approximation theory in the spaces
$L^{p}(\mathbb{R}^{d},d\mu_{k})$, namely Theorem \ref{thm1} and Corollary
\ref{thm2}. In Section~\ref{sec6}, we derive estimates of the modulus of
smoothness of a function $f$ via the fractional powers of the Dunkl Laplacian
of entire functions $f_{\sigma}$ and $\eta_{j}f$.

Pitt- and Kellog-type estimates given in Theorems \ref{thm5} and \ref{thm-pitt}
and the results on Besov--Dunkl spaces
%,  \ref{thm-besov-pitt}, and \textbf{Theorem~\ref{thm-besov}
are proved in Section~\ref{sec7}.
% Эта работа посвящена близким вопросам,
%связанным с Besov-type spaces and integrability for the Dunkl transform, однако
%в ней не применяется модуль непрерывности для задания классов Бесова.

\bigskip
\section{Elements of Dunkl harmonic analysis}\label{sec2}

In this section, we recall the basic notation and results of the Dunkl harmonic
analysis (see, e.g., \cite{Ros99,Ros03a,GorIvaTik19}).

The Dunkl kernel $e_{k}(x, y)=E_{k}(x, iy)$ is a unique
solution of the system
\[
\nabla_{k}f(x)=iyf(x),\quad f(0)=1,
\]
where $\nabla_{k}=(D_{1,k},\dots,D_{d,k})$ is the Dunkl gradient. The Dunkl
kernel plays the role of a generalized exponential function and its properties
are similar to those of the classical exponential function $e_{0}(x, y)=e^{i(x,
y)}$. Several basic properties follow from the integral
representation~\cite{Ros99}
\[
e_k(x,y)=\int_{\mathbb{R}^d}e^{i(\xi,y)}\,d\mu_x^k(\xi),
\]
where $\mu_x^k$ is a probability Borel measure and
$\operatorname{supp}\mu_x^k\subset \operatorname{co}{}(\{gx\colon g\in
G(R)\})$. In parti\-cular,
\[%\begin{equation}\label{exponent}
|e_k(x,y)|\leq 1,\quad e_{k}(x, y)=e_{k}(y, x), \quad e_{k}(-x, y)=\overline{e_{k}(x, y)}.
\]%\end{equation}

For $f\in L^{1}(\mathbb{R}^{d},d\mu_{k})$, the Dunkl transform is defined by
\[
\mathcal{F}_{k}(f)(y)=\int_{\mathbb{R}^{d}}f(x)\overline{e_{k}(x,
y)}\,d\mu_{k}(x).
\]
For $k\equiv0$ we recover % $\mathcal{F}_{0}$ is
 the classical Fourier transform $\mathcal{F}$.
%We also note that $\mathcal{F}_k(e^{-|\Cdot|^2/2})(y)=e^{-|y|^2/2}$ and $\mathcal{F}_{k}^{-1}(f)(x)=\mathcal{F}_{k}(f)(-x)$.

As usual, by $\mathcal{A}_k $ we denote the Wiener class
\[%\begin{equation}\label{zvzv}
\mathcal{A}_k=\bigl\{f\in L^{1}(\mathbb{R}^{d},d\mu_{k})\cap C_b(\mathbb{R}^d)\colon
\mathcal{F}_{k}(f)\in L^{1}(\mathbb{R}^{d},d\mu_{k})\bigr\}.
\]%\end{equation}
Several basic properties of the Dunkl transform are  collected in the following result.

\begin{proposition}[\cite{Ros03a}]\label{prop2.1}
\textup{(1)} For $f\in L^{1}(\mathbb{R}^{d},d\mu_{k})$, one has $\mathcal{F}_{k}(f)\in
C_0(\mathbb{R}^d)$.

\smallbreak
\textup{(2)} If $f\in\mathcal{A}_k$, then the following pointwise inversion
formula holds:
\[
f(x)=\int_{\mathbb{R}^{d}}\mathcal{F}_{k}(f)(y)e_{k}(x, y)\,d\mu_{k}(y).
\]

\textup{(3)} The Dunkl transform leaves the Schwartz space
$\mathcal{S}(\mathbb{R}^d)$ invariant.

\smallbreak
\textup{(4)} The Dunkl transform extends to a unitary self-adjoint operator in
$L^{2}(\mathbb{R}^{d},d\mu_{k})$,
$\mathcal{F}_{k}^{-1}(f)(x)=\mathcal{F}_{k}(f)(-x)$.
\end{proposition}

Let $\mathbb{S}^{d-1}=\{x'\in\mathbb{R}^d\colon |x'|=1\}$ be the Euclidean
sphere, and let $d\sigma_k(x')=a_kv_k(x')\,dx'$ be the probability measure on
$\mathbb{S}^{d-1}$. The following formula is well known
\cite[Corollary~2.5]{Ros03b}:
\[%\begin{equation}\label{averaging}
\int_{\mathbb{S}^{d-1}}e_{k}(x, ty')\,d\sigma_k(y')=j_{\lambda_k}(t|x|),\quad
x\in \mathbb{R}^{d},
\]%\end{equation}
where $\lambda_{k}$ is given in \eqref{Weighteddimension} and
$j_{\lambda}(t)=2^{\lambda}\Gamma(\lambda+1)t^{-\lambda}J_{\lambda}(t)$ is the
normalized Bessel function.

Let $y\in \mathbb{R}^d$ be given. R\"{o}sler \cite{Ros98} defined a
generalized translation operator $\tau^y$ in $L^{2}(\mathbb{R}^{d},d\mu_{k})$
by the equation
\[
\mathcal{F}_k(\tau^yf)(z)=e_k(y, z)\mathcal{F}_k(f)(z).
\]
Since $|e_k(y, z)|\leq 1$, $\|\tau^y\|_{2\to 2}\leq1$.
The operator $\tau^yf$ is not positive %in common case
and it remains an open question whether $\tau^yf$ is an $L^p$-bounded operator for $p\neq 2$.

%Let
%\begin{equation}\label{Weighteddimension}
%\lambda_k=d/2-1+\sum_{a\in R_+}k(a),\quad d_k=2(\lambda_k+1).
%\end{equation}
%The number $d_k$ plays the role of the
%generalized dimension of the space $(\mathbb{R}^d,d\mu_k)$.
%We have $\lambda_k\geq-1/2$ and, moreover, $\lambda_k=-1/2$ only if $d=1$ and
%$k\equiv0$. In what follows we assume that $\lambda_k>-1/2$ and $d_k>1$.

Let $t\ge 0$. In \cite{GorIvaTik19}, we have recently defined the different generalized translation
operator $T^t$ in $L^{2}(\mathbb{R}^{d},d\mu_{k})$ by
\[%\begin{equation}\label{multiplier}
\mathcal{F}_k(T^tf)(y)=j_{\lambda_k}(t|y|)\mathcal{F}_k(f)(y).
\]%\end{equation}
In light of $|j_{\lambda_k}(t)|\leq 1$, we have $\|T^t\|_{2\to 2}\leq1$.

We now list some basic properties of the operator $T^t$, $t\in\mathbb{R}_+$.

\begin{proposition}[\cite{GorIvaTik19,Ros03b}]%\label{prop2.3}
\textup{(1)} If $f\in\mathcal{A}_k$, then %pointwise
\[
T^tf(x)=\int_{\mathbb{R}^d}j_{\lambda_k}(t|y|)e_k(x,
y)\mathcal{F}_k(f)(y)\,d\mu_k(y)=\int_{\mathbb{S}^{d-1}}\tau^{ty'}f(x)\,d\sigma_k(y').
\]

\smallbreak
\textup{(2)} The operator $T^t$ is positive. If $f\in C_b(\mathbb{R}^d)$, then
\[%\begin{equation}\label{represent}
T^tf(x)=\int_{\mathbb{R}^d}f(z)\,d\sigma_{x,t}^k(z)\in C_b(\mathbb{R}_+\times\mathbb{R}^d),
\]%\end{equation}
where $\sigma_{x,t}^k$ is a probability Borel measure such that
$\operatorname{supp}\sigma_{x,t}^k\subset \bigcup_{g\in G}B_{t}(gx)$. In
particular, $T^t1=1$.

\smallbreak
\textup{(3)} If $f\in \mathcal{S}(\mathbb{R}^d)$, $1\leq p\leq\infty$, then
$\|T^tf\|_{p,d\mu_k}\leq \|f\|_{p,d\mu_k}$ and the operator $T^t$ can be
extended to $L^{p}(\mathbb{R}^{d},d\mu_{k})$ with preservation of the norm.
\end{proposition}

Note that for $k\equiv0$, $T^tf(x)$ coincides the usual spherical mean
$\int_{\mathbb{S}^{d-1}}f({x+ty'})\,d\sigma_{0}(y')$.

Let $g(y)=g_0(|y|)$ be a radial function. Thangavelu and
Xu~\cite{ThaXu05} defined the convolution
\begin{equation}\label{convolution1}
(f\ast_{k}g)(x)=\int_{\mathbb{R}^d}f(y)\tau^{x}g(-y)\,d\mu_{k}(y).
\end{equation}

\begin{proposition}[\cite{GorIvaTik19,ThaXu05}]%\label{prop2.4}
\textup{(1)} If $f\in\mathcal{A}_k$, $g\in L^{1}_{\mathrm{rad}}(\mathbb{R}^d,d\mu_{k})$, then
\[%\begin{equation*}\label{eq15}
(f\ast_{k}g)(x)=\int_{\mathbb{R}^d}\tau^{-y}f(x)g(y)\,d\mu_k(y)\in \mathcal{A}_k
\]%\end{equation*}
and
\begin{equation*}
\mathcal{F}_k(f\ast_{k}g)(y)=\mathcal{F}_k(f)(y)\mathcal{F}_k(g)(y),\quad y\in\mathbb{R}^d.
\end{equation*}

\smallbreak
\textup{(2)} {\it Let $1\leq p\leq\infty$. If $f\in L^{p}(\mathbb{R}^{d},d\mu_{k})$,
$g\in L^{1}_{\mathrm{rad}}(\mathbb{R}^d,d\mu_{k})$, then
$(f\ast_{k}g)\in L^{p}(\mathbb{R}^{d},d\mu_{k})$, and}
\[%\begin{equation}\label{ineq}
\|(f\ast_{k}g)\|_{p}\leq \|f\|_{p}\,\|g\|_{1}.
\]%\end{equation}
\end{proposition}

%Помимо Young's convolution inequality \eqref{ineq} для преобразования Данкля
%справедливы следующие классические неравенства:

We also mention the following Hausdorff--Young and Hardy--Littlewood type
inequalities.
\begin{proposition}[\cite{AASS09}]\label{prop-class}
One has
\begin{equation}\label{hy-ineq}
\|\mathcal{F}_{k}(f)\|_{p'}\le \|f\|_{p},\quad 1\le p\le 2,\quad
\frac{1}{p}+\frac{1}{p'}=1,
\end{equation}
and
\[
\bigl\||x|^{d_{k}(1-2/p)}\mathcal{F}_{k}(f)(x)\bigr\|_{p}\lesssim \|f\|_{p},\quad 1<p\le
2,
\]
where $d_{k}$ is the generalized dimension defined by \eqref{Weighteddimension}.
\end{proposition}

We will use the following known Hardy's inequality:
\begin{equation}\label{har-1}
\sum_{j=0}^{\infty}2^{-j\gamma}\biggl(\sum_{l=0}^{j}A_{l}\biggr)^{\vartheta}\asymp
\sum_{j=0}^{\infty}2^{-j\gamma}A_{j}^{\vartheta},
\end{equation}
\begin{equation}\label{har-2}
\sum_{j=0}^{\infty}2^{j\gamma}\biggl(\sum_{l=j}^{\infty}A_{l}\biggr)^{\vartheta}\asymp
\sum_{j=0}^{\infty}2^{j\gamma}A_{j}^{\vartheta},
\end{equation}
where  $A_{j}\ge 0$, $\gamma>0$, and $0<\vartheta\le \infty$ (with the standard modification for  $\vartheta=\infty$); see e.g. \cite{jaen}.

\smallskip
\section{Smoothness characteristics and the~$K$-functional}\label{sec3}

\subsection{Bernstein's class of entire functions}\label{subsec-bern}
Let $\mathbb{C}^d$ be the complex Euclidean space of $d$ dimensions,
$z=(z_1,\dots,z_d)\in \mathbb{C}^d$, $|z|=\sqrt{\sum_{i=1}^{d}|z_{i}|^{2}}$,
and $\mathrm{Im}\,z=(\mathrm{Im}\,z_1,\dots,\mathrm{Im}\,z_d)$.

 For $\sigma>0$ we define the Bernstein class
$\mathcal{B}_{p,k}^{\sigma}$ of entire function
of exponential spherical type at most $\sigma$. %s, namely, %and $\widetilde{\mathcal{B}}_{p,k}^{\sigma}$.
We say that a function $f\in \mathcal{B}_{p,k}^{\sigma}$ if $f\in
L^{p}(\mathbb{R}^{d},d\mu_{k})$ is such that its analytic continuation to
$\mathbb{C}^d$ satisfies
\begin{equation*}
|f(z)|\leq C_{\varepsilon}e^{(\sigma+\varepsilon)|z|},\quad
\forall\,\varepsilon>0,\ \forall\,z\in \mathbb{C}.
\end{equation*}
The smallest $\sigma=\sigma_{f}$ in this inequality is called a spherical type of $f$.
%In other words, the class $\mathcal{B}_{p,k}^{\sigma}$ is the collection
%of all entire functions of spherical type at most $\sigma$.

In \cite{GorIvaTik19}, we proved that functions $f\in
\mathcal{B}_{p,k}^{\sigma}$ satisfy
\[
|f(z)|\leq Ce^{\sigma|\mathrm{Im}\,z|},\quad \forall\,z\in \mathbb{C}^d.
\]
%We say that a function $f\in \widetilde{\mathcal{B}}_{p,k}^\sigma$ if $f\in
%L^{p}(\mathbb{R}^{d},d\mu_{k})$ is such that its analytic continuation to
%$\mathbb{C}^d$ satisfies
%\[
%|f(z)|\leq Ce^{\sigma|\mathrm{Im}\,z|},\quad \forall\,z\in \mathbb{C}^d.
%\]
%It turns out that both classes coincide \cite{GorIvaTik19}.
Moreover, the following Paley--Wiener type characterization holds true.

\begin{proposition}[\cite{GorIvaTik19}]\label{prop4}
A function $f\in \mathcal{B}_{p,k}^{\sigma}$, $1\leq p<\infty$, if and only if
\[
f\in L^{p}(\mathbb{R}^{d},d\mu_{k})\cap C_b(\mathbb{R}^d)\quad \text{and}\quad
\operatorname{supp}\mathcal{F}_k(f)\subset B_\sigma(0).
\]
\end{proposition}

The Dunkl transform $\mathcal{F}_k(f)$ in Proposition \ref{prop4} is understood
as a function for $1\leq p\leq 2$ and as a tempered distribution for $p>2$.

%Let\[E_{\sigma}(f)_{p}=\inf\{\|f-g\|_{p}\colon g\in \mathcal{B}_{p,k}^{\sigma}\}\]
%be the best approximation of a function $f\in
%L^{p}(\mathbb{R}^{d},d\mu_{k})$ by entire functions of spherical exponential
%type at most $\sigma$.

\subsection{Lizorkin and Sobolev spaces}\label{subsec-lapl}
Now we define the fractional power of the Dunkl Laplacian. Let
\[
\Phi_k=\biggl\{f\in \mathcal{S}(\mathbb{R}^d)\colon
\int_{\mathbb{R}^{d}}x_{1}^{\alpha_{1}}\ldots
x_{d}^{\alpha_{d}}f(x)\,d\mu_k(x)=0,\ \alpha\in \mathbb{Z}^d_+\biggr\}
\]
be the weighted Lizorkin space (see \cite{Liz63,GorIvaTik17,SamKilMar93}) and set
\[
\Psi_k=\{\mathcal{F}_{k}(f)\colon f\in\Phi_k\}.
\]
%At the spaces $\Phi_k$ and $\Psi_k$ we consider the same convergence as in
%$\mathcal{S}(\mathbb{R}^{d})$. {\bf S: Ubrat?}

\begin{proposition}[\cite{GorIvaTik17}]
\textup{(1)} The spaces $\Phi_k$ and $\Psi_k$ are closed in the topology of
$\mathcal{S}(\mathbb{R}^d)$.

\textup{(2)} The space $\Phi_k$ is dense in $L^{p}(\mathbb{R}^{d},d\mu_{k})$
for $1\leq p<\infty$.

\textup{(3)} One has
\[
\Psi_k\equiv \Psi_0=\{\mathcal{F}(f)\colon
f\in\Phi_0\}=\biggl\{f\in\mathcal{S}(\mathbb{R}^{d})\colon
\frac{\partial^{\alpha_{1}+\ldots+\alpha_{d}}}{\partial
x_{1}^{\alpha_{1}}\ldots \partial x_{d}^{\alpha_{d}}}\,f(0)=0,\ \alpha\in
\mathbb{Z}^d_+\biggr\}.
\]
\end{proposition}

Now we will use some auxiliary results from our paper \cite{GorIva19}. Let
$\Phi_k'$ and $\Psi_k'$ be the spaces of distributions on $\Phi_k$ and $\Psi_k$
respectively. We have $\mathcal{S}'(\mathbb{R}^{d})\subset\Phi_k'$,
$\mathcal{S}'(\mathbb{R}^{d})\subset\Psi_k'$, and
$\Phi_k'=\mathcal{S}'(\mathbb{R}^{d})/\Pi$,
$\Psi_k'=\mathcal{S}'(\mathbb{R}^{d})/\mathcal{F}_k(\Pi)$, where $\Pi$ stands
for the set of all polynomials of $d$ variables. We can multiply distributions
from $\Psi_k'$ on functions from
\[
C^{\infty}_{\Pi}(\mathbb{R}^d\setminus\{0\})=\bigl\{|x|^sf(x)\colon f\in
C^{\infty}_{\Pi}(\mathbb{R}^d),\ s\in\mathbb{R}\bigr\},
\]
where $C^{\infty}_{\Pi}(\mathbb{R}^d)$ is the space of infinitely differentiable
functions whose derivatives have polynomial growth at infinity.

%Next, using multipliers from $C^{\infty}_{\Pi}(\mathbb{R}^d\setminus\{0\})$ and
%Dunkl transform we can define several distributions. {\bf S: ne jasno!}
Next, using Dunkl multipliers we can define the following distributions.
Let $r>0$. We define the fractional power of the Dunkl Laplacian for
$\varphi\in\Phi_k$ as follows
\[
(-\Delta_k)^{r/2}\varphi=\mathcal{F}_k^{-1}(|\Cdot|^{r}\mathcal{F}_k(\varphi))
=\mathcal{F}_k(|\Cdot|^{r}\mathcal{F}_k^{-1}(\varphi))\in\Phi_k.
\]
(see also \cite{MPS20}).
By definition, for $f\in\Phi_k'$ the distribution $(-\Delta_k)^{r/2}f\in\Phi_k'$ is % defined by
\[%\begin{equation}\label{fractional}
((-\Delta_k)^{r/2}f,\varphi)=(f,(-\Delta_k)^{r/2}\varphi)=
(f,\mathcal{F}_k^{-1}(|\Cdot|^{r}\mathcal{F}_k(\varphi))),\quad \varphi\in \Phi_k.
\]%\end{equation}
By $W^{r}_{p,k}$, $1\leq p\leq\infty$, we denote the Sobolev space, that is,
\[
W^{r}_{p,k}=\{f\in L^{p}(\mathbb{R}^{d},d\mu_{k})\colon (-\Delta_k)^{r/2}f\in
L^{p}(\mathbb{R}^{d},d\mu_{k})\}
\]
equipped with the norm
\[
\|f\|_{W^{r}_{p,k}}=\|f\|_{p}+\bigl\|(-\Delta_k)^{r/2}f\bigr\|_{p}.
\]

\subsection{Basic definitions in the distributional sense}
Let us now define the direct and inverse Dunkl transforms
$\mathcal{F}_k$, $\mathcal{F}_k^{-1}$, generalized translation operators
$\tau^y$, $T^t$, and convolution $(f\ast_{k}g)$ for
distributions.

For $f\in\Phi_k'$ the direct Dunkl transform
$\mathcal{F}_k(f)\in\Psi_k'$ is defined by
\[
(\mathcal{F}_k(f),\psi)=(f,\mathcal{F}_k(\psi)), \quad\psi\in\Psi_k.
\]
Similarly, for $g\in\Psi_k'$ the inverse Dunkl transform
$\mathcal{F}_k^{-1}(g)\in\Phi_k'$ is defined by
\[
(\mathcal{F}_k^{-1}(g),\varphi)=(g,\mathcal{F}_k^{-1}(\varphi)), \quad\varphi\in\Phi_k.
\]
We have
\[
\mathcal{F}_k^{-1}(\mathcal{F}_k(f))=f,\quad\mathcal{F}_k(\mathcal{F}_k^{-1}(g))=g,
\quad f\in\Phi_k',\ g\in\Psi_k'.
\]
Note that $f=g$ in $\Phi_k'$ if and only if $\mathcal{F}_k(f)=\mathcal{F}_k(g)$ in $\Psi_k'$.

For $f\in\Phi_k'$ the generalized translation operators
$\tau^yf,T^tf\in\Phi_k'$ are given respectively~by
\begin{equation*}
\begin{gathered}
(\tau^yf,\varphi)=(f,\tau^{-y}\varphi)=(f,\mathcal{F}_k^{-1}(e_k(-y,{\cdot\,})\mathcal{F}_k(\varphi))),\quad
y\in\mathbb{R}^{d},\\
(T^tf,\varphi)=(f,T^t\varphi)=(f,\mathcal{F}_k^{-1}(j_{\lambda_k}(t|\Cdot|)
\mathcal{F}_k(\varphi))),\quad
t\in\mathbb{R}_+,
\end{gathered}
\end{equation*}
where $\varphi\in \Phi_k$. Moreover, the following
 %For the Dunkl transform of the considered operators and their compositions we have the following
 equalities are valid:
\begin{equation*}%\label{distributiontransform}
\begin{gathered}
\mathcal{F}_k((-\Delta_k)^{r/2}f)=|\Cdot|^{r}\mathcal{F}_k(f),\quad
\mathcal{F}_k(\tau^yf)=e_k(y,{\cdot\,})\mathcal{F}_k(f),\\
\mathcal{F}_k((-\Delta_k)^{r/2}\tau^yf)=|\Cdot|^{r}e_k(y,{\cdot\,})\mathcal{F}_k(f), \quad
\mathcal{F}_k(T^tf)=j_{\lambda_k}(t|\Cdot|)\mathcal{F}_k(f),\\
\mathcal{F}_k((-\Delta_k)^{r/2}T^tf)=|\Cdot|^{r}j_{\lambda_k}(t|\Cdot|)
\mathcal{F}_k(f),\\ \mathcal{F}_k(T^t(\tau^yf))=j_{\lambda_k}(t|\Cdot|)e_k(y,{\cdot\,})\mathcal{F}_k(f).
\end{gathered}
\end{equation*}
In particular, this implies the commutativity of considered operators.

Let $\varphi\in\Phi_k$ and $\varphi^{-}(y)=\varphi(-y)$. We say that
$f\in\Phi_k'$ is even if $(f,\varphi^{-})=(f,\varphi)$. Similarly we define
even $g\in\Psi_k'$. Note that $f\in\Phi_k'$ is even if and only if
$\mathcal{F}(f)\in\Psi_k'$ is even.

Let $N_k$ be a set of all even $f\in\Phi_k'$ such that $\mathcal{F}_k(f)\in
C^{\infty}_{\Pi}(\mathbb{R}^d\setminus\{0\})$. For $f\in N_k$ and
$\varphi\in\Phi_k$ we set
\[
(f\ast_{k}\varphi)(x)=(\tau^{x}f,\varphi^{-})=(f,\tau^{-x}\varphi^{-}).
\]
If $g\in N_k$ and $\varphi\in\Phi_k$, then $(g\ast_{k}\varphi)\in\Phi_k$ and
\[ (g\ast_{k}\varphi)(x)=\mathcal{F}_k^{-1}(\mathcal{F}_k(g)\mathcal{F}_k(\varphi))(x),\quad
\mathcal{F}_k(g\ast_{k}\varphi)(y)=\mathcal{F}_k(g)(y)\mathcal{F}_k(\varphi)(y).
\]
Therefore, we can define the convolution $(f\ast_{k}g)\in\Phi_k'$ for $f\in\Phi_k'$ and $g\in N_k$ as follows
\begin{equation}\label{convolution2}
((f\ast_{k}g),\varphi)=(f,(g\ast_{k}\varphi)),\quad \varphi\in\Phi_k.
\end{equation}
Moreover, we remark that \begin{equation}\label{commutation}
\begin{gathered}
\mathcal{F}_k(f\ast_{k}g)=\mathcal{F}_k(g)\mathcal{F}_k(f),\quad
(-\Delta_k)^{r/2}(f\ast_{k}g)=((-\Delta_k)^{r/2}f\ast_{k}g),\\
(f\ast_{k}(g_1\ast_{k}g_2))=((f\ast_{k}g_1)\ast_{k}g_2)=(f\ast_{k}(g_2\ast_{k}g_1))=((f\ast_{k}g_2)\ast_{k}g_1).
\end{gathered}
\end{equation}

The next result establishes the interrelation between the convolutions given by \eqref{convolution1} and \eqref{convolution2}.
\begin{proposition}[\cite{GorIva19}]\label{prop5}
If $f\in L^{p}(\mathbb{R}^{d},d\mu_{k})$, $g\in
L^{1}_{\mathrm{rad}}(\mathbb{R}^d,d\mu_{k})$, and $\mathcal{F}_k(g)\in N_k$, then
the convolutions given by \eqref{convolution1} and \eqref{convolution2}
 %of these functions
  coincide.
\end{proposition}

\subsection{Moduli of smoothness and $K$-functionals}\label{subsec-mod}
The $K$-functional for the couple $(L^{p}(\mathbb{R}^{d},d\mu_{k}), W^{r}_{p,k})$ is defined in the usual way: for $t>0$
\[
K_{r}(f,t)_{p}=K_{r}(f,t; L^{p}(\mathbb{R}^{d},d\mu_{k}), W^{r}_{p,k})_{p}=\inf\{\|f-g\|_{p}+t^{r}\bigl\|(-\Delta_k)^{r/2}g\bigr\|_{p}\colon
g\in W^{r}_{p,k}\}.
\]
Since
\begin{equation*}%\label{K-inequality}
|K_{r}(f_1,t)_{p}-K_{r}(f_2,t)_{p}|\leq \|f_1-f_2\|_{p},
\end{equation*} for any $f\in L^{p}(\mathbb{R}^{d},d\mu_{k})$, one has
$\lim_{t\to 0}K_{r}(f,t)_{p}=0$.
The monotonicity property of the $K$-functional is given by
\begin{equation}\label{K-property2}
K_{r}(f,\lambda t)_{p}\leq \max\{1,\,\lambda^{r}\}K_{r}(f,t)_{p}.
\end{equation}

By definition,
\[
\mathcal{R}_{r}(f,t)_{p}=\inf
\Bigl\{\|f-g\|_{p}+t^{r}\bigl\|(-\Delta_k)^{r/2}g\bigr\|_{p}\colon g\in
\mathcal{B}_{p,k}^{1/t}\Bigr\}
\]
is the realization of the $K$-functional $K_{r}(f,t)_{p}$. Moreover, define
\[
\mathcal{R}^*_{r}(f,t)_{p}=\|f-g^*\|_{p}+t^{r}\bigl\|(-\Delta_k)^{r/2}g^*\bigr\|_{p},
\]
where $g^*\in \mathcal{B}_{p,k}^{1/t}$ is the best or near best approximant for
$f$ in $L^{p}(\mathbb{R}^d, d\mu_k)$. The realization of the $K$-functional was
introduced in \cite{DitHriIva95, Hri}, where its
 importance in the approximation theory was shown.

\begin{proposition}[\cite{GorIva19}]\label{prop6}
{\it Suppose $t> 0$, $1\leq p\leq\infty$, and $r>0$, then, for any} $f\in
L^{p}(\mathbb{R}^{d},d\mu_{k})$,
\[
\mathcal{R}_{r}(f,t)_{p}\asymp \mathcal{R}^*_{r}(f,t)_{p} \asymp
K_{r}(f,t)_{p}\asymp\omega_r(f,t)_{p}.
\]
\end{proposition}
For the case of integer $r$, see \cite[Cor.~2.3, Th.~3.1]{DD04}  ($k\equiv 0$)
and \cite{GorIvaTik19} ($k({\cdot})\ge 0$). For the case of fractional moduli,
see \cite{kol, sim} ($k\equiv 0$). The discussion on various ways to define moduli
of smoothness can be found in \cite[Sec.~6]{GorIvaTik19}.

Let $\omega_m(f,\delta)_{p}$ denote the modulus of smoothness of order $m>0$ of
a function $f\in L^{p}(\mathbb{R}^{d},d\mu_{k})$, i.e.,
\[
\omega_m(f,\delta)_{p}=\sup_{0<t\leq\delta}\|\varDelta_t^mf(x)\|_{p},\quad
\delta>0,
\]
where
\begin{equation}\label{difference}
\varDelta_t^mf(x)=(I-T^t)^{m/2}f(x)=
\sum_{s=0}^{\infty}(-1)^s\binom{m/2}{s}(T^{t})^sf(x)
\end{equation}
and $I$ stands for the identical operator. The difference \eqref{difference}
coincides with the classical fractional difference for the translation operator
$T^{t}f(x)=f(x+t)$ and corresponds to the usual definition of the fractional
modulus of smoothness, see, e.g., \cite{but, SamKilMar93}.

Now we give several basic properties of the modulus of smoothness and the
difference \eqref{difference} (see \cite{GorIva19}):
\begin{equation}\label{omega-norma}
\lim_{\delta\to 0+0}\omega_m(f,\delta)_{p}=0,\quad
\omega_m(f,\delta)_{p}\lesssim\|f\|_{p},\quad \delta>0,
\end{equation}
\[%\begin{equation}\label{ineqmodulus}
\omega_m(f_1+f_2,\delta)_{p}\leq
\omega_m(f_1,\delta)_{p}+\omega_m(f_2,\delta)_{p},
\]
\begin{equation}\label{omega-lambda}
\omega_m(f,\lambda\delta)_{p}\lesssim\max (1,\lambda^{2m})
\omega_m(f,\delta)_{p},\quad \lambda>0,
\end{equation}
\begin{equation}\label{omega-equiv}
\omega_m(f,\delta)_{p}\asymp \|\varDelta_\delta^mf\|_{p},\quad \delta>0,\quad
d_{k}>1,
\end{equation}
\begin{equation}\label{transformdifference1}
\mathcal{F}_k(\varDelta_t^mf)=(1-j_{\lambda_{k}}(t|\Cdot|))^{m/2}\mathcal{F}_k(f),\quad f\in \Phi_k'.
\end{equation}

We conclude this section by presenting the Bernstein inequality  in the Dunkl setting.

\begin{proposition}[\cite{GorIva19}]\label{prop7}
If $\sigma,r>0$, $1 \le p\le \infty$, $f\in \mathcal{B}_{p,k}^{\sigma}$, then
\begin{equation}\label{bernstein}
\bigl\|(-\Delta_k)^{r/2}f\bigr\|_{p}\lesssim \sigma^{r}\|f\|_{p},
\end{equation}
\begin{equation}\label{ber-omega}
\omega_r(f,\delta)_{p}\lesssim
\delta^{r}\bigl\|(-\Delta_k)^{r/2}f\bigr\|_{p}\lesssim
(\delta\sigma)^{r}\|f\|_{p}.
\end{equation}
\end{proposition}

\bigskip
\section{Littlewood--Paley-type inequalities}\label{sec4}

Recall that $\eta\in \mathcal{S}_{\mathrm{rad}}(\mathbb{R}^d)$ such that
$\eta(x)=1$ if $|x|\leq 1/2$, $\eta(x)>0$ if $|x|<1$, and $\eta(x)=0$ if
$|x|\geq 1$.
Set $\theta(x)=\eta(x)-\eta(2x)$,
\begin{equation*}
\eta_j(x)=\eta(2^{-j}x),\quad \theta_j(x)=\theta(2^{-j}x)=\eta_j(x)-\eta_{j-1}(x), \quad j\in \mathbb{Z}.
\end{equation*}
Let $\eta_jf$, $\theta_jf$ be multiplier linear operators, defined by relations
%\[
%\mathcal{F}_{k}(\eta_jf)(y)=\eta_j(y)\mathcal{F}_{k}(f)(y),\quad
%\mathcal{F}_{k}(\theta_jf)(y)=\theta_j(y)\mathcal{F}_{k}(f)(y),
%\]
\[
\mathcal{F}_{k}(\eta_j f)=\mathcal{F}_{k}(f)\eta_j,\qquad
\mathcal{F}_{k}(\theta_j f)=\mathcal{F}_{k}(f)\theta_j,
%\theta_jf=\mathcal{F}_{k}^{-1}(\mathcal{F}_{k}(f)\theta_j),
\]
respectively.
% \textbf{!!!Отметим, что в силу Proposition~\ref{prop2.4} они являются операторами свертки. Ne ponimaju etot vopros}
We have %For them, the following properties are valid:
\begin{equation}\label{properties_1}
\begin{gathered}
\operatorname{supp}\eta_j\subset B_{2^j},\quad
\operatorname{supp}\theta_j\subset B_{2^j}(0)\setminus B_{2^{j-2}}(0),\\
%\{x\in\mathbb{R}^d\colon\, 2^{j-2}\le |x|\le 2^j\},\\
\eta_j(\eta_if)=\eta_i(\eta_jf)=\eta_jf,\quad j<i,
\end{gathered}
\end{equation}
and for $f\in L^p(\mathbb{R}^d, d\mu_k)$ (see \cite{GorIvaTik19,GorIva19})
\begin{equation}\label{properties_2}
\begin{gathered}
\eta_{j}f,\,\theta_{j}f\in \mathcal{B}_{p,k}^{2^j},\quad \|\eta_jf\|_p\lesssim
\|f\|_p,\quad \|\theta_jf\|_p\lesssim \|f\|_p,\\
\|f-\eta_jf\|_p\lesssim E_{2^{j-1}}(f)_p\lesssim \|f-\eta_{j-1}f\|_p,\\
\|\theta_jf\|_p\lesssim \|f-\eta_{j-1}f\|_p+\|f-\eta_jf\|_p\lesssim E_{2^{j-2}}(f)_p.
\end{gathered}
\end{equation}
Moreover,
\begin{equation}\label{properties_3}
\begin{gathered}
\eta_0(x)+\sum_{j=1}^{\infty}\theta_j(x)=1,\ x\in\mathbb{R}^d\setminus
\{0\},\quad \eta_0f+\sum_{j=1}^{\infty}\theta_jf=f,\\
\int_{\mathbb{R}^d}\theta_if\,\theta_jf\,d\mu_k=0,\quad |i-j|\ge 2,
\end{gathered}
\end{equation}
%where $I$ is the identity operator.
and for any function $f\in L_p(\mathbb{R}^d, d\mu_k)$ the series
$\eta_0f+\sum_{j=1}^{\infty}\theta_jf$ converges to $f$ in $L^p(\mathbb{R}^d,
d\mu_k)$.

\begin{lemma}\label{lem1}
Suppose $1\le p<\infty$ and $r>0$, then $\mathcal{S}(\mathbb{R}^d)$ is dense in
$W^{r}_{p,k}$.
\end{lemma}

\begin{proof}
Setting
$\mathcal{B}_{p,k}^{\infty}=\bigcup_{\sigma>0}\mathcal{B}_{p,k}^{\sigma}$, in
virtue of \eqref{commutation}, \eqref{properties_2}, Propositions \ref{prop5}
and \ref{prop7}, for any $f\in W^{r}_{p,k}$ we derive that $\eta_{j}f\in
\mathcal{B}_{p,k}^{2^j}$ and
\begin{align*}
(-\Delta_k)^{r/2}\eta_{j}f&=(-\Delta_k)^{r/2}(f\ast_{k}\mathcal{F}_{k}(\eta_j))\\&
=((-\Delta_k)^{r/2}f)\ast_{k}\mathcal{F}_{k}(\eta_j))=\eta_{j}((-\Delta_k)^{r/2}f)\in
\mathcal{B}_{p,k}^{2^j}.
\end{align*}
Since, by Bernstein's inequality \eqref{bernstein}, the embedding
$\mathcal{B}_{p,k}^{2^j}\subset W^{r}_{p,k}$ holds, \eqref{properties_2}
implies
\[
\|f-\eta_jf\|_p\lesssim E_{2^{j-1}}(f)_p, \quad
\bigl\|(-\Delta_k)^{r/2}f-(-\Delta_k)^{r/2}(\eta_jf)\bigr\|_p\lesssim
E_{2^{j-1}}((-\Delta_k)^{r/2}f)_p.
\]
Hence, $\mathcal{B}_{p,k}^{\infty}$ is dense in $W^{r}_{p,k}$.

Let $f\in \mathcal{B}_{p,k}^{\sigma}$, $\delta,\varepsilon>0$ and suppose that
$\psi\in \mathcal{S}(\mathbb{R}^d)$ is an entire function of exponential type 1
such that $\psi(0)=1$. Let $\psi_{\delta}(x)=\psi(\delta x)$. Then inequality
\eqref{bernstein} and the Nikolskii inequality \cite{GorIva19} yield that
$f\psi_{\delta} \in\mathcal{S}(\mathbb{R}^d)\cap
\mathcal{B}_{p,k}^{\sigma+\delta}$. Choose $R>0$ and $0<\delta<1$ so that
\[
\int_{|x|\ge R}|f(x)|^p\,d\mu_k(x)<\varepsilon^{p},\quad |1-\psi_{\delta}
(x)|<\varepsilon\ \text{for}\ |x|\le R.
\]
Then we have
\begin{align*}
\int_{\mathbb{R}^d}|f(x)-f(x)\psi_{\delta}(x)|^p\,d\mu_k(x)&\le (1+\|\psi\|_{\infty})^p\int_{|x|\ge R}|f(x)|^p\,d\mu_k(x)\\&+\varepsilon^{p}\int_{|x|\le R}|f(x)|^p\,d\mu_k(x)\le (1+\|\psi\|_{\infty}+\|f\|_p)^p\varepsilon^p.
\end{align*}
Using again Bernstein's inequality \eqref{bernstein}, we finally obtain
\begin{align*}
\bigl\|(-\Delta_k)^{r/2}(f-f\psi_{\delta})\bigr\|_p&\le
c(k)(\sigma+1)^r\|(f-f\psi_{\delta})\|_p\\&\le
c(k)(\sigma+1)^r(1+\|\psi\|_{\infty}+\|f\|_p)\varepsilon.
\end{align*}
\end{proof}

Now we establish the desired version of the Littlewood--Paley inequalities. To
prove it, we follow the same reasoning as those in the papers
\cite{DaiDit05,DaiDitTik08,DaiWan10} (see also \cite[Chapter~7]{DaiXu15}). We
will also use  \cite{VelYes19}.

\begin{theorem}%\label{thm6}
Let $1<p<\infty$, $r\ge 0$. If $f\in W^r_{p,k}$, then
\begin{equation}\label{littlewood_1}
\bigl\|(-\Delta_k)^{r/2}f\bigr\|_{p}\asymp
\biggl\|\biggl(\sum_{j\in\mathbb{Z}}2^{2rj}|\theta_jf|^2\biggr)^{1/2}\biggr\|_{p}
\end{equation}
and
\begin{equation}\label{littlewood_2}
\bigl\|(-\Delta_k)^{r/2}f\bigr\|_{p}\asymp
\biggl\|\biggl(\bigl|(-\Delta_k)^{r/2}\eta_0f\bigr|^2+\sum_{j=1}^{\infty}2^{2rj}|\theta_jf|^2\biggr)^{1/2}\biggr\|_{p}.
\end{equation} Moreover,
\begin{equation}\label{littlewood_3}
\bigl\|(-\Delta_k)^{r/2}f\bigr\|_{p}\lesssim
\biggl\|\biggl(|\eta_0f|^2+\sum_{j=1}^{\infty}2^{2rj}|\theta_jf|^2\biggr)^{1/2}\biggr\|_{p}.
\end{equation}
\end{theorem}

\begin{proof}
For $f\in \mathcal{S}(\mathbb{R}^d)$ and $r\ge 0$, equivalence
\eqref{littlewood_1} was proved in \cite[Proposition~4.5]{VelYes19}. By Lemma
\ref{lem1}, $\mathcal{S}(\mathbb{R}^d)$ is dense in $W^r_{p,k}$ and hence,
equivalence~\eqref{littlewood_1} is valid for any $f\in W^r_{p,k}$.

Inequality \eqref{bernstein} implies $\eta_0f\in W^r_{p,k}$. Applying the equality $\eta_0f(x)=\sum_{j=-\infty}^0\theta_jf(x)$ and \eqref{littlewood_1}, we have
\[
\bigl\|(-\Delta_k)^{r/2}\eta_0f\bigr\|_{p}\asymp
\biggl\|\biggl(\sum_{j=-\infty}^{0}2^{2rj}|\theta_jf|^2\biggr)^{1/2}\biggr\|_{p}.
\]
Hence,
\begin{align*}
\bigl\|(-\Delta_k)^{r/2}f\bigr\|_{p}&\asymp
\biggl\|\biggl(\sum_{j\in\mathbb{Z}}2^{2rj}|\theta_jf|^2\biggr)^{1/2}\biggr\|_{p}\\
&\asymp
\biggl\|\biggl(\sum_{j=-\infty}^{0}2^{2rj}|\theta_jf|^2\biggr)^{1/2}\Bigr\|_{p}+
\biggl\|\biggl(\sum_{j=1}^{\infty}2^{2rj}|\theta_jf|^2\biggr)^{1/2}\biggr\|_{p}\\
&\asymp\bigl\|(-\Delta_k)^{r/2}\eta_0f\bigr\|_{p}+\biggl\|\biggl(\sum_{j=1}^{\infty}2^{2rj}|\theta_jf|^2\biggr)^{1/2}\biggr\|_{p}\\
&\asymp
\biggl\|\biggl(\bigl|(-\Delta_k)^{r/2}\eta_0f\bigr|^2+\sum_{j=1}^{\infty}2^{2rj}|\theta_jf|^2\biggr)^{1/2}\biggr\|_{p},
\end{align*}
that is, \eqref{littlewood_2} is shown. Since
$\bigl\|(-\Delta_k)^{r/2}\eta_0f\bigr\|_p\lesssim\|\eta_0f\|_p $, we obtain
\eqref{littlewood_3} from \eqref{littlewood_2}.
\end{proof}

To prove Theorem \ref{thm-pitt}~(2) we will use a more general version of lower estimate in \eqref{littlewood_1} with $r=0$.

\begin{lemma}[\cite{VelYes19}]\label{lem-V}
Let $\varphi\in \mathcal{S}_{\mathrm{rad}}(\mathbb{R}^d)$,
$\operatorname{supp}\varphi\subset \{\alpha\le |x|\le \beta\}$,
$0<\alpha<\beta$, $\varphi_{j}(x)=\varphi(2^{-j}x)$,
and $\varphi_{j}f=\mathcal{F}_{k}^{-1}(\mathcal{F}_{k}(f)\varphi_{j})$. Then for
$1<p<\infty$
we have
\[
\biggl\|\biggl(\sum_{j\in
\mathbb{Z}}|\varphi_{j}f|^{2}\biggr)^{1/2}\biggr\|_{p}\lesssim \|f\|_{p}.
\]
\end{lemma}

Note that in \cite{VelYes19} this result was shown for  $\alpha=1/2$, $\beta=2$. The general case  is similar.
%Так как кратность покрытия $\mathbb{R}_+$ отрезками $[2^j\alpha,2^j\beta]$,
%$j\in\mathbb{Z}$, не превосходит $1+2\bigl[\log_{2}\frac{\beta}{\alpha}\bigr]$,
%<то общий случай разбирается аналогично.

To prove Theorem \ref{thm2}, we will also need the following result.
\begin{corollary}\label{cor1}
If $f\in L^p(\mathbb{R}^d, d\mu_k)$, $1<p<\infty$, $q=\min (p,2)$, then
\[
\|f\|_p\lesssim
\biggl(\|\eta_0f\|_p^p+\sum_{j=1}^{\infty}\|\theta_jf\|_p^q\biggr)^{1/q}.
\]
\end{corollary}

\begin{proof}
The proof is carried out as the corresponding result in \cite{DaiDit05}. We give it for completeness.
If $1<p\le 2$, then $q=p$ and
\begin{align*}
\|f\|_p&\lesssim
\biggl\|\biggl(|\eta_0f|^2+\sum_{j=1}^{\infty}|\theta_jf|^2\biggr)^{1/2}\biggr\|_{p}\le
\biggl\|\biggl(|\eta_0f|^p+\sum_{j=1}^{\infty}|\theta_jf|^p\biggr)^{1/p}\biggr\|_{p}\\&=
\biggl(\|\eta_0f\|_p^p+\sum_{j=1}^{\infty}\|\theta_jf\|_p^p\biggr)^{1/p}.
\end{align*}
If $2\le p<\infty$, $r=p/2\ge 1$, then Minkowski's inequality implies
\begin{align*}
\|f\|_p&\lesssim
\biggl\|\biggl(|\eta_0f|^2+\sum_{j=1}^{\infty}|\theta_jf|^2\biggr)^{1/2}\biggr\|_{p}=
\biggl\||\eta_0f|^2+\sum_{j=1}^{\infty}|\theta_jf|^2\biggr\|_{r}^{1/2}\\&\le
\biggl(\bigl\||\eta_0f|^2\bigr\|_r+\sum_{j=1}^{\infty}\||\theta_jf|^2\|_r\biggr)^{1/2}=
\biggl(\|\eta_0f\|_p^2+\sum_{j=1}^{\infty}\|\theta_jf\|_p^2\biggr)^{1/2}.
\end{align*}

\end{proof}

\bigskip
\section{Proofs of Theorem~\ref{thm1} and Corollary~\ref{thm2}}% and \ref{thm2}}
\label{sec5}

%\subsection*{Proof of Theorem \ref{thm1}}
\begin{proof}[Proof of Theorem \ref{thm1}]
Following the corresponding proof in \cite{DaiDitTik08}, since $E_{j}(f)_{p}$,
$\omega_r(f,1/j)_{p}$, $K_{r}(f,1/j)_{p}$ are all monotonic in $j$, by
\eqref{eq3}, we can equivalently write inequality \eqref{eq4} in the form
\begin{equation}\label{ineq}
J=2^{-rn}\biggl(\sum_{j=0}^{n}2^{srj}E_{2^j}^s(f)_{p}\biggr)^{1/s}\lesssim
K_r(f,2^{-n})_{p}.
\end{equation}

Set $g_n=\eta_{n-1}f$. Applying \eqref{properties_1} and \eqref{littlewood_2} with $r=0$, we have
$E_{2^n}(f)_p\le \|f-g_n\|_{p}$, $E_{2^n}(g_n)_p=0$,
and for $0\le j<n$
\[
E_{2^j}(g_n)_p\le\|\eta_ng_n-\eta_jg_n\|_{p}=
\biggl\|\sum_{l=j+1}^n\theta_lg_n\biggr\|_{p}\lesssim
\biggl\|\biggl(\sum_{l=j+1}^n|\theta_lg_n|^2\biggr)^{1/2}\biggr\|_{p}.
\]
Hence,
\begin{align}\label{commonestimate}
J&\le2^{-rn}\biggl(\sum_{j=0}^{n-1}2^{srj}E_{2^j}^s(f-g_n)_{p}\biggr)^{1/s}{+}
2^{-rn}\biggl(\sum_{j=0}^{n-1}2^{srj}E_{2^j}^s(g_n)_{p}\biggr)^{1/s}\notag\\
&\lesssim \|f-g_n\|_{p}+2^{-rn}\biggl(\sum_{j=0}^{n-1}2^{srj}
\biggl\|\biggl(\sum_{l=j+1}^n|\theta_lg_n|^2\biggr)^{1/2}\biggr\|_{p}^s\biggr)^{1/s}.
\end{align}

Let $1<p\le 2$ and $s=2$. Using \eqref{commonestimate}, the inequality $p/2\le 1$,
\eqref{littlewood_2} with $r>0$, and Proposition~\ref{prop6}, we obtain
\begin{align*}\label{commonestimate}
J&\lesssim \|f-g_n\|_{p}+2^{-rn}\biggl(\sum_{j=0}^{n-1}2^{2rj}
\biggl\|\sum_{l=j+1}^n|\theta_lg_n|^2\biggr\|_{p/2}\biggr)^{1/2}\\
&\lesssim\|f-g_n\|_{p}+2^{-rn}\biggl(\biggl\|\sum_{j=0}^{n-1}2^{2rj}
\sum_{l=j+1}^n|\theta_lg_n|^2\biggr\|_{p/2}\biggr)^{1/2}\\
&=\|f-g_n\|_{p}+2^{-rn}\biggl(\biggl\|\sum_{l=1}^{n}|\theta_lg_n|^2
\sum_{j=0}^{l-1}2^{2rj}\biggr\|_{p/2}\biggr)^{1/2}\\ &\lesssim
\|f-g_n\|_{p}+2^{-rn}\biggl(\biggl\|\sum_{l=1}^{n}2^{2rl}|\theta_lg_n|^2
\biggr\|_{p/2}\biggr)^{1/2}\\
&=\|f-g_n\|_{p}+2^{-rn}\biggl\|\biggl(\sum_{l=1}^{n}2^{2rl}|\theta_lg_n|^2\biggr)^{1/2}
\biggr\|_{p}\\ &\lesssim
\|f-g_n\|_{p}+2^{-rn}\bigl\|(-\Delta_k)^{r/2}g_n\bigr\|_{p}\lesssim
K_r(f,2^{-n})_p.
\end{align*}
Thus, we verified \eqref{ineq} for $1<p\le 2$.

Let $2<p<\infty$ and $s=p$. Applying the duality between $\ell_{p/2}^n$ and
$\ell_q^n$, where $q=(p/2)^{'}$, we can write
\[
\sum_{j=0}^{n-1}2^{prj}\biggl(\sum_{l=j+1}^n|\theta_lg_n|^2(x)\biggr)^{p/2}
=\biggl(\sum_{j=0}^{n-1}2^{prj}a_j(x)\sum_{l=j+1}^n|\theta_lg_n|^2(x)\biggr)^{p/2},
\]
where $\sum_{j=0}^{n-1}2^{prj}a_j^q(x)=1$.
Using this, we derive
\begin{align*}\label{commonestimate}
L&=\int_{\mathbb{R}^d}\sum_{j=0}^{n-1}2^{prj}\biggl(\sum_{l=j+1}^n|\theta_lg_n|^2(x)\biggr)^{p/2}\,d\mu_k(x)\\
&=\int_{\mathbb{R}^d}\biggl(\sum_{j=0}^{n-1}2^{prj}a_j(x)\sum_{l=j+1}^n|\theta_lg_n|^2(x)\biggr)^{p/2}\,d\mu_k(x)\\
&=\int_{\mathbb{R}^d}\biggl(\sum_{l=1}^{n}|\theta_lg_n|^2(x)\sum_{j=0}^{l-1}2^{prj}a_j(x)\biggr)^{p/2}\,d\mu_k(x).
\end{align*}
Applying H\"{o}lder's inequality and \eqref{littlewood_2}, we obtain
\begin{align*}
L&\le\int_{\mathbb{R}^d}\biggl(\sum_{l=1}^{n}|\theta_lg_n|^2(x)\biggl(\sum_{j=0}^{l-1}2^{prj}\biggr)^{2/p}
\biggl(\sum_{j=0}^{n-1}2^{prj}a_j^q(x)\biggr)^{1/q}\biggr)^{p/2}\,d\mu_k(x)\\
&\lesssim
\int_{\mathbb{R}^d}\biggl(\sum_{l=1}^{n}2^{2rj}|\theta_lg_n|^2(x)\biggr)^{2/p}\,d\mu_k(x)=
\biggl\|\biggl(\sum_{l=1}^{n}2^{2rl}|\theta_lg_n|^2\biggr)^{1/2}\biggr\|_{p}^p\\
&\lesssim\bigl\|(-\Delta_k)^{r/2}g_n\bigr\|_{p}^p.
\end{align*}
Hence, from \eqref{commonestimate}, we get
\[
J\lesssim\|f-g_n\|_{p}+2^{-rn}\bigl\|(-\Delta_k)^{r/2}g_n\bigr\|_{p}\lesssim
K_r(f,2^{-n})_p,
\]
that is, \eqref{ineq} follows. Thus, \eqref{eq4} is proved.

%We will follow the paper \cite{DaiDit05}.
 Since the Sobolev space is dense in
$L^{p}(\mathbb{R}^{d},d\mu_{k})$, we can assume that $f\in W^r_{p,k}$ and write
inequality \eqref{eq6} in the form
\begin{equation}\label{ineq_3}
K_r(f,2^{-n})_{p}\lesssim
2^{-rn}\biggl\{\biggl(\sum_{j=0}^{n}2^{qrj}E_{2^j}^q(f)_{p}\biggr)^{1/q}+\|f\|_p\biggr\}.
\end{equation}
Taking into account Proposition~\ref{prop6}, \eqref{littlewood_2}, Corollary~\ref{cor1},
\eqref{commutation}, and \eqref{properties_1}--\eqref{properties_3}, this gives
\begin{align*}
K_r(f,2^{-n})_{p}&\lesssim
\|f-\eta_nf\|_p+2^{-rn}\bigl\|(-\Delta_k)^{r/2}\eta_nf\bigr\|_p\\ &\lesssim
E_{2^{n-1}}(f)_p+2^{-rn}\biggl\|\biggl(\bigl|\eta_0((-\Delta_k)^{r/2}\eta_nf)\bigr|^2+
\sum_{j=1}^{\infty}\bigl|\theta_j((-\Delta_k)^{r/2}\eta_nf)\bigr|^2\biggr)\biggr\|_p\\
&\lesssim E_{2^{n-1}}(f)_p+2^{-rn}
\biggl(\bigl\|\eta_0((-\Delta_k)^{r/2}\eta_nf)\bigr\|_p^q+
\sum_{j=1}^{\infty}\bigl\|\theta_j((-\Delta_k)^{r/2}\eta_nf)\bigr\|_p^q\biggr)^{1/q}\\
&\lesssim E_{2^{n-1}}(f)_p+2^{-rn}
\biggl(\bigl\|(-\Delta_k)^{r/2}\eta_0f\bigr\|_p^q+
\sum_{j=1}^{n-1}\bigl\|(-\Delta_k)^{r/2}\theta_jf\bigr\|_p^q\\
&+\bigl\|\eta_n((-\Delta_k)^{r/2}\theta_nf)\bigr\|_p^q+
\bigl\|\eta_n((-\Delta_k)^{r/2}\theta_{n+1}f)\bigr\|_p^q\biggr)^{1/q}.
\end{align*}
Bernstein's inequality \eqref{bernstein} yields
\begin{align*}
K_r(f,2^{-n})_{p}&\lesssim
 E_{2^{n-1}}(f)_p+2^{-rn}\biggl(\|\eta_0f\|_p^q+\sum_{j=1}^{n+1}2^{qrj}\|\theta_jf\|_p^q\biggr)^{1/q}\\
&\lesssim
2^{-rn}\biggl\{\biggl(\sum_{j=0}^{n}2^{qrj}E_{2^j}^q(f)_{p}\biggr)^{1/q}+\|f\|_p\biggr\}.
\end{align*}
\end{proof}
%Theorem \ref{thm1} isproved.\qed

\begin{proof}[Proof of Corollary \ref{thm2}]
 Inequality \eqref{eq5} can be equivalently written as follows
\[
J=2^{-rn}\biggl(\sum_{j=0}^{n}2^{srj}\omega_m^s(f,2^{-j})_{p}\biggr)^{1/s}\lesssim
\omega_r(f,2^{-n})_{p}+2^{-rn}\|f\|_{p}.
\]
Using \eqref{eq2},  Hardy's inequality \eqref{har-1}
\[
\sum_{j=0}^{n}2^{-\gamma j}\biggl(\sum_{l=0}^{j}A_l\biggr)^{q} \lesssim
\sum_{j=0}^{n}2^{-\gamma j}A_j^{q},
\]
and \eqref{eq4}, we have
\begin{align*}
J&\lesssim 2^{-rn}\biggl(\sum_{j=0}^{n}2^{-(m-r)sj}\biggl(\sum_{l=0}^{j}2^{ml}E_{2^{l}}^s(f)_{p}\biggr)^{s}\biggr)^{1/s}+2^{-rn}\|f\|_{p}\\
&\lesssim 2^{-rn}\biggl(\sum_{j=0}^{n}2^{rsj}E_{2^{j}}^s(f)_{p}\biggr)^{1/s}+2^{-rn}\|f\|_{p}\\
&\lesssim\omega_r(f,2^{-n})_{p}+2^{-rn}\|f\|_{p}.
\end{align*}
Inequality \eqref{eq7} follows from \eqref{ineq_3} and Jackson's inequality \eqref{eq1}. \end{proof}

%\section{Connection between smoothness of functions and smoothness of~approximation processes}\label{sec6}
\section{Proofs of Theorem \ref{thm3}}\label{sec6}

%\begin{proof}[Proof of Theorem \ref{thm3}]
\begin{proof}[Proof of Theorem \ref{thm3} for the best approximants]
It follows from \cite[Theorem1]{LimSma91} that given a function $f\in
L^{p}(\mathbb{R}^{d},d\mu_{k})$, $1<p<\infty$, for any entire function $g\in
\mathcal{B}_{p,k}^{\sigma}$ one has
\begin{equation}\label{strongunique_1}
\|f-f_{\sigma}\|_{p}^s\le \|f-g\|_{p}^s-A\|g-f_{\sigma}\|_{p}^s,\quad s=\max(p,2),
\end{equation}
\begin{equation}\label{strongunique_2}
\|f-f_{\sigma}\|_{p}^q\le \|f-g\|_{p}^q-B\|g-f_{\sigma}\|_{p}^q,\quad q=\min(p,2),
\end{equation}
where $f_{\sigma}\in \mathcal{B}_{p,k}^{\sigma}$ is the best approximant of $f$
and positive constants $A,B$ are independent of $f,f_{\sigma},g$.

Following similar arguments as those in \cite{KolTik19}, let us prove the
left-hand side inequality in \eqref{eq8}. Using Hardy's inequality
\eqref{har-1}
\begin{equation}\label{Hardy}
\sum_{j=n}^{\infty}2^{-rj}\biggl(\sum_{l=n}^{j}A_l\biggr)^{q}\lesssim
\sum_{j=n}^{\infty}2^{-rj}A_j^{q},
\end{equation}
inequality \eqref{strongunique_1}, Proposition~\ref{prop6}, and \eqref{eq3}, we
derive that
\begin{align*}
&\quad
\sum_{j=n+1}^{\infty}2^{-srj}\bigl\|(-\Delta_k)^{r/2}f_{2^j}\bigr\|_{p}^s\\
&=\sum_{j=n+1}^{\infty}2^{-srj}\biggl\|\sum_{l=n+1}^{j}(-\Delta_k)^{r/2}(f_{2^l}-f_{2^{l-1}})+(-\Delta_k)^{r/2}f_{2^n}\biggr\|_{p}^s\\
&\lesssim\sum_{j=n+1}^{\infty}2^{-srj}\biggl\|\sum_{l=n+1}^{j}\bigl((-\Delta_k)^{r/2}f_{2^l}-
(-\Delta_k)^{r/2}f_{2^{l-1}}\bigr)\biggr\|_{p}^s+
2^{-srn}\bigl\|(-\Delta_k)^{r/2}f_{2^n}\bigr\|_{p}^s\\
&\lesssim\sum_{j=n+1}^{\infty}2^{-srj}\biggl(\sum_{l=n+1}^{j}\bigl\|(-\Delta_k)^{r/2}f_{2^l}-(-\Delta_k)^{r/2}f_{2^{l-1}}\bigr\|_{p}\biggr)^s+
2^{-srn}\bigl\|(-\Delta_k)^{r/2}f_{2^n}\bigr\|_{p}^s\\
&\lesssim\sum_{j=n+1}^{\infty}2^{-srj}\bigl\|(-\Delta_k)^{r/2}f_{2^j}-(-\Delta_k)^{r/2}f_{2^{j-1}}\bigr\|_{p}^s+
2^{-srn}\bigl\|(-\Delta_k)^{r/2}f_{2^n}\bigr\|_{p}^s.
\end{align*}
Then
Bernstein's inequality \eqref{bernstein} implies
\begin{align*}
&\quad
\sum_{j=n+1}^{\infty}2^{-srj}\bigl\|(-\Delta_k)^{r/2}f_{2^j}\bigr\|_{p}^s
\lesssim\sum_{j=n+1}^{\infty}\|f_{2^j}-f_{2^{j-1}}\|_{p}^s+2^{-srn}\bigl\|(-\Delta_k)^{r/2}f_{2^n}\bigr\|_{p}^s\\
&\lesssim\frac{1}{A}\sum_{j=n+1}^{\infty}(\|f-f_{2^{j-1}}\|_{p}-\|f-f_{2^{j}}\|_{p})^s+2^{-srn}\bigl\|(-\Delta_k)^{r/2}f_{2^n}\bigr\|_{p}^s\\
&\lesssim\|f-f_{2^{n}}\|_{p}^s+2^{-srn}\bigl\|(-\Delta_k)^{r/2}f_{2^n}\bigr\|_{p}^s\lesssim
K_r^s(2^{-n},f)_p\lesssim \omega_r^s(2^{-n},f)_p.
\end{align*}
To show the right-hand side inequality in \eqref{eq8}, by \eqref{eq3} and
\eqref{K-property2}, we have
\begin{align*}
\omega_r^q(2^{-n},f)_p&\lesssim K_r^q(2^{-n},f)_p\\
&\lesssim\|f-f_{2^{n+1}}\|_{p}^q+2^{-qrn}\bigl\|(-\Delta_k)^{r/2}f_{2^{n+1}}\bigr\|_{p}^q\\
&\lesssim\sum_{j=n+2}^{\infty}\bigl(\|f-f_{2^{j-1}}\|_{p}^q-\|f-f_{2^{j}}\|_{p}^q\bigr)+
2^{-qrn}\bigl\|(-\Delta_k)^{r/2}f_{2^{n+1}}\bigr\|_{p}^q\\
&\lesssim\sum_{j=n+2}^{\infty}(\|f-f_{2^{j-1}}(f_{2^{j}})\|_{p}^q-\|f-f_{2^{j}}\|_{p}^q)+
2^{-qrn}\bigl\|(-\Delta_k)^{r/2}f_{2^{n+1}}\bigr\|_{p}^q.
\end{align*}
Using \eqref{strongunique_2} and the following Jackson inequality \cite{GorIva19}
\[
E_{\sigma}(f)_{p}\lesssim \sigma^{-r}\,
\bigl\|(-\Delta_k)^{r/2}f\bigr\|_{p},\quad f\in W^{r}_{p,k},\ 1\le p\le\infty,\ \sigma, r>0,
\]
we obtain
\begin{align*}
\omega_r^q(2^{-n},f)_p& \lesssim
\sum_{j=n+2}^{\infty}\|f_{2^{j}}-f_{2^{j-1}}(f_{2^{j}})\|_{p}^q+
2^{-qrn}\bigl\|(-\Delta_k)^{r/2}f_{2^{n+1}}\bigr\|_{p}^q\\ &\lesssim
\sum_{j=n+2}^{\infty}E_{2^{j-1}}(f_{2^j})_{p}+
2^{-qrn}\bigl\|(-\Delta_k)^{r/2}f_{2^{n+1}}\bigr\|_{p}^q\\& \lesssim
\sum_{j=n+1}^{\infty}2^{-qrj}\bigl\|(-\Delta_k)^{r/2}f_{2^j}\bigr\|_{p},
\end{align*}
%The second inequality in \eqref{eq8} is proved.
completing the proof.
\end{proof}

\begin{proof}[Proof of Theorem \ref{thm3} for the de la Vall\'{e}e Poussin
type operators]
We will show that for $s=\max(p,2)$ and $q=\min(p,2)$,
%best approximants $f_{2^{j}}$ or the de la Vall\'{e}e Poussin
%type operators~$\eta_{j}f$.
\begin{align}
\biggl(\sum_{j=n+1}^{\infty}2^{-srj}\bigl\|(-\Delta_k)^{r/2}\eta_{j}f\bigr\|_{p}^s\biggr)^{1/s}&\lesssim
\omega_r(f,2^{-n})_{p}\notag\\&\lesssim
\biggl(\sum_{j=n+1}^{\infty}2^{-qrj}\bigl\|(-\Delta_k)^{r/2}\eta_{j}f\bigr\|_{p}^q\biggr)^{1/q}.
\label{eq9}
\end{align}

To obtain the left-hand side estimate, we have %Let $s=\max(p,2)\ge 2$.
% us prove the inequality equivalent to the left-hand side inequality in \eqref{eq9}
%\begin{equation}\label{ineq_4}
%\biggl(\sum_{j=n+1}^{\infty}2^{-srj}\bigl\|(-\Delta_k)^{r/2}\eta_{2^j}f\bigr\|_{p}^s\biggr)^{1/s}\lesssim K_r(f,2^{-n})_{p},
%\end{equation}
%where $s=\max(p,2)\ge 2$.
\begin{align}\label{ineq_5}
%&\quad
\sum_{j=n+1}^{\infty}2^{-srj}\bigl\|(-\Delta_k)^{r/2}\eta_{j}f\bigr\|_{p}^s
%\notag\\
 &\lesssim
\sum_{j=n+1}^{\infty}2^{-srj}\bigl\|(-\Delta_k)^{r/2}(\eta_{j}-\eta_n)f\bigr\|_{p}^s
\\&+\bigl\|(-\Delta_k)^{r/2}\eta_nf)\bigr\|_{p}^s \notag
%\\&
=:J+\bigl\|(-\Delta_k)^{r/2}\eta_nf)\bigr\|_{p}^s.
\end{align}
In light of \eqref{littlewood_3}, we obtain
\begin{align*}
J&\lesssim
\sum_{j=n+1}^{\infty}2^{-srj}\biggl\|\biggl(|\eta_0((\eta_{j}-\eta_n)f)|^2+\sum_{l=1}^{\infty}2^{2rl}
|\theta_l((\eta_{j}-\eta_n)f)|^2\biggr)^{1/2}\biggr\|_{p}^s\\
&=\sum_{j=n+1}^{\infty}2^{-srj}\biggl\|\biggl(\sum_{l=n}^{j+1}2^{2rl}|\theta_l((\eta_{j}-\eta_n)f)|^2\biggr)^{1/2}\biggr\|_{p}^s.
\end{align*}
If $l=n,n+1$, $j\ge n+1$, then \eqref{properties_1} and \eqref{properties_2}
yield
\[
\|\theta_l((\eta_j-\eta_n)f)\|_p\lesssim \|(\eta_j-\eta_n)f\|_p=\|\eta_j(f-\eta_nf)\|_p\lesssim \|f-\eta_nf\|_p.
\]
If $l=j,j+1$, $j\ge n+1$, then
\[
\|\theta_l((\eta_j-\eta_n)f)\|_p=\|\eta_j(\theta_l(f-\eta_nf))\|_p\lesssim \|\theta_l(f-\eta_nf)\|_p.
\]
If $n+2\le l\le j-1$, then $\theta_l((\eta_j-\eta_n)f)=\theta_l(f-\eta_nf)$.

Hence,
\[
J\lesssim \sum_{j=n+1}^{\infty}2^{-srj}\biggl\|\biggl(\sum_{l=n}^{j+1}2^{2rl}|\theta_l(f-\eta_nf)|^2\biggr)^{1/2}\biggr\|_{p}^s+\|f-\eta_nf\|_p^s.
\]
Taking into account Minkowski's inequality with $\frac{p}{s}\le 1$, Hardy's
inequality \eqref{Hardy} and equivalence \eqref{littlewood_2}, we obtain
\begin{align*}
&\quad
\sum_{j=n+1}^{\infty}2^{-srj}\biggl\|\biggl(\sum_{l=n}^{j+1}2^{2rl}|\theta_l(f-\eta_nf)|^2\biggr)^{1/2}\biggr\|_{p}^s\\
&\le
\biggl\|\sum_{j=n+1}^{\infty}2^{-srj}\biggl(\sum_{l=n}^{j+1}2^{2rl}|\theta_l(f-\eta_nf)|^2\biggr)^{s/2}\biggr\|_{p/s}\\
&\lesssim\biggl\|\sum_{j=n}^{\infty}|\theta_l(f-\eta_nf)|^{s}\biggr\|_{p/s}
\lesssim\biggl\|\biggl(\sum_{j=n}^{\infty}|\theta_l(f-\eta_nf)|^{2}\biggr)^{1/2}\biggr\|_{p}^{\tau}\lesssim \|f-\eta_nf\|_p^s.
\end{align*}
Therefore, $J\lesssim \|f-\eta_nf\|_p^s$ and by \eqref{ineq_5}, we arrive at
\[
\biggl(\sum_{j=n+1}^{\infty}2^{-srj}\bigl\|(-\Delta_k)^{r/2}\eta_{2^j}f\bigr\|_{p}^s\biggr)^{1/s}\lesssim
\|f-\eta_nf\|_p+\bigl\|(-\Delta_k)^{r/2}\eta_nf)\bigr\|_{p}\lesssim K_r(f,2^{-n})_{p}.
\]
%Thus, we arrive at the left-hand side inequality in \eqref{eq9}.
Proposition \ref{prop6} concludes the proof of the left-hand side inequality in \eqref{eq9}.

To verify the right-hand side inequality in \eqref{eq9}, it suffices to prove that
\begin{equation}\label{ineq_6}
K_r(f,2^{-n})_{p}\lesssim
\biggl(\sum_{j=n+1}^{\infty}2^{-qrj}\bigl\|(-\Delta_k)^{r/2}\eta_{2^j}f\bigr\|_{p}^q\biggr)^{1/q}.
%,\quad q=\min(p,2).
\end{equation}
By Proposition~\ref{prop6} and \eqref{properties_2}, we have
\begin{equation}\label{ineq_7}
K_r(f,2^{-n})_{p}^q\lesssim
\|f-\eta_nf\|_p^q+2^{-qrn}\bigl\|(-\Delta_k)^{r/2}\eta_nf)\bigr\|_{p}^q.
\end{equation}
Using \eqref{littlewood_2}, the inequality $|\theta_j(f-\eta_nf)|^2\le
2(|\theta_jf|^2+|\theta_j(\eta_nf)|^2)$, \eqref{properties_2} and the
equalities $\theta_j(\eta_nf)=0$ for $j\ge n+2$, $\theta_j(f-\eta_nf)=0$ for
$j\le n-1$, we obtain
\[
\|\theta_j(f-\eta_nf)\|_p\lesssim \|\theta_jf\|_p,\quad j=n,\,n+1,
\]
and
\begin{align*}
\|f-\eta_nf\|_p^q&\lesssim \biggl\|\biggl(\sum_{j=n}^{\infty}|\theta_j(f-\eta_nf)|^2\biggr)^{1/2}\biggr\|_{p}^q\\
&\lesssim \biggl\|\biggl(\sum_{j=n}^{\infty}|\theta_jf|^2\biggr)^{1/2}\biggr\|_{p}^q\lesssim
\biggl\|\biggl(\sum_{j=n}^{\infty}|\theta_jf|^q\biggr)^{1/q}\biggr\|_{p}^q\\
&=\biggl\|\biggl(\sum_{j=n}^{\infty}2^{-rqj}(|\theta_jf|^2\,2^{2rj})^{q/2}\biggr)^{1/q}\biggr\|_{p}^q\\
&\lesssim \biggl\|\sum_{j=n}^{\infty}2^{-rqj}\biggl(\sum_{l=n}^{j+2}2^{2rl}|\theta_l(\eta_{j+1}f)|^2\biggr)^{q/2}\biggr\|_{p/q}\\
&\lesssim \sum_{j=n}^{\infty}2^{-rqj}\biggl\|\biggl(\sum_{l=1}^{\infty}2^{2rl}|\theta_l(\eta_{j+1}f)|^2\biggr)^{1/2}\biggr\|_{p}^q.
\end{align*}
In view of \eqref{littlewood_2},
\[
\|f-\eta_nf\|_p^q\lesssim\sum_{j=n}^{\infty}2^{-rqj}\bigl\|(-\Delta_k)^{r/2}(\eta_{j+1}f)\bigr\|_{p}^q\lesssim
\sum_{j=n}^{\infty}2^{-rqj}\bigl\|(-\Delta_k)^{r/2}(\eta_{j}f)\bigr\|_{p}^q.
\]
This and \eqref{ineq_7} imply \eqref{ineq_6}.
\end{proof}

%\begin{proof}[Proof of Proposition~\ref{prop-bes}]
%\end{proof}

\bigskip
\section{Proofs of Theorems~\ref{thm5}--\ref{thm-besov-pitt}}\label{sec7}

\begin{proof}[Proof of Theorem \ref{thm-pitt}]
First we obtain Pitt-type estimates \eqref{pitt1} and \eqref{pitt2}.
% \eqref{kell2} and XXXX
 For $1<p\le 2$ and $p\le q\le p'$ the inequality
\[
\bigl\||x|^{d_{k}(1/p'-1/q)}\mathcal{F}_{k}(f)\bigr\|_{q}\lesssim
\|f\|_{p}%,\quad
\]
immediately follows from Proposition~\ref{prop-class} and the interpolation theorem
\cite[Theorem~2]{Ste56}.

Let now $2\le p<\infty$ and $p'\le q\le p$.
By Proposition~\ref{prop2.1}~(4) to obtain estimate \eqref{pitt2}, we prove that
\begin{equation}\label{pi}
\|\mathcal{F}_{k}(f)\|_{p}\lesssim
\bigl\||x|^{d_{k}(1/p'-1/q)}f\bigr\|_{q}.
\end{equation}
%Using the fact that $\mathcal{F}_{k}$ is self-adjoint,
Proposition~\ref{prop-class} implies the following Hardy--Littlewood type
inequality: %for $2\le p<\infty$:
\begin{align}
\|\mathcal{F}_{k}(f)\|_{p}&=\sup_{\|g\|_{p'}\le 1}
\biggl|\int_{\mathbb{R}^{d}}\mathcal{F}_{k}(f)g\,d\mu_{k}\biggr|=
\sup_{\|g\|_{p'}\le 1}
\biggl|\int_{\mathbb{R}^{d}}f\mathcal{F}_{k}(g)\,d\mu_{k}\biggr|\notag\\
&=\sup_{\|g\|_{p'}\le 1}
\biggl|\int_{\mathbb{R}^{d}}|x|^{d_{k}(1-2/p)}f\,|x|^{d_{k}(1-2/p')}
\mathcal{F}_{k}(g)\,d\mu_{k}\biggr|\notag\\ &\le
\bigl\||x|^{d_{k}(1-2/p)}f\bigr\|_{p}\sup_{\|g\|_{p'}\le
1}\bigl\||x|^{d_{k}(1-2/p')}\mathcal{F}_{k}(g)\bigr\|_{p'}\lesssim
\bigl\||x|^{d_{k}(1-2/p)}f\bigr\|_{p}. \label{dhl-ineq}
\end{align}
As usual, we first obtain this inequality for $f,g\in \mathcal{S}(\mathbb{R}^d)$ %on the Schwartz space
and then we use
density arguments to consider the general case $|\Cdot|^{d_{k}(1-2/p)}f\in
L^{p}(\mathbb{R}^{d},d\mu_{k})$.
Interpolating  between $\|\mathcal{F}_{k}(f)\|_{p}\le \|f\|_{p'}$ and \eqref{dhl-ineq}, we
arrive at inequality \eqref{pi}.

\smallbreak
Second we derive Kellogg-type inequalities \eqref{kell1} and \eqref{kell2}.
Let $1<p\le 2$ and $f\in L^p(\mathbb{R}^d,d\mu_k)$. To verify \eqref{kell1},
%We will show that
%\begin{equation}\label{k-1}
%\biggl(\sum_{j\in
%\mathbb{Z}}\,\bigl\|\mathcal{F}_{k}(f)\chi_{j}\bigr\|_{p'}^{2}\biggr)^{1/2}\lesssim
%\|f\|_{p}.
%\end{equation}
 we will use Lemma \ref{lem-V} for  $\varphi$ with support in %с носителем в кольце
  the annulus
  $\{1/2\le |x|\le 3 \}$ and such that $\varphi(x)=1$ for $1\le |x|\le 2$.
Then
\begin{equation}\label{phi-chi}
\chi_{j}(x)\equiv\chi_{\{2^{j}\le |x|<2^{j+1}\}}(x)\le \varphi_{j}(x),\quad x\in \mathbb{R}^{d}.
\end{equation}
%где, напомним, $\chi_{j}=\chi_{\{2^{j}\le |x|<2^{j+1}\}}$.

Putting  $A_{j}=|\varphi_{j}f|^{p}$ and $\alpha=2/p$,
Lemma ~\ref{lem-V} gives
\[
\|f\|_{p}\gtrsim \biggl(\int_{\mathbb{R}^{d}}\biggl(\sum_{j\in
\mathbb{Z}}|\varphi_{j}f|^{2}\biggr)^{p/2}\,d\mu_{k}\biggr)^{1/p}=
\biggl(\int_{\mathbb{R}^{d}}\biggl(\sum_{j\in
\mathbb{Z}}A_{j}^{\alpha}\biggr)^{1/\alpha}\,d\mu_{k}\biggr)^{1/p}.
\]
Making use of Minkowsky's inequality %, we establish %Используя неравенство Минковского
\[
\biggl(\sum_{j\in
\mathbb{Z}}\biggl(\int_{\mathbb{R}^{d}}A_{j}\,d\mu_{k}\biggr)^{\alpha}\biggr)^{1/\alpha}\le
\int_{\mathbb{R}^{d}}\biggl(\sum_{j\in
\mathbb{Z}}A_{j}^{\alpha}\biggr)^{1/\alpha}\,d\mu_{k},
\]
we derive that \[
\|f\|_{p}\gtrsim \biggl(\sum_{j\in
\mathbb{Z}}\biggl(\int_{\mathbb{R}^{d}}A_{j}\,d\mu_{k}\biggr)^{\alpha}\biggr)^{1/(\alpha p)}=
\biggl(\sum_{j\in
\mathbb{Z}}\|\varphi_{j}f\|_{p}^{2}\biggr)^{1/2}.
\]
Applying  the Hausdorff--Young inequality \eqref{hy-ineq} for $\varphi_{j}f$ and   \eqref{phi-chi}, we have

\[
\|\varphi_{j}f\|_{p}\gtrsim\|\mathcal{F}_{k}(\varphi_{j}f)\bigr\|_{p'}= \|\mathcal{F}_{k}(f)\varphi_{j}\bigr\|_{p'}\ge
\|\mathcal{F}_{k}(f)\chi_{j}\|_{p'},
\]
Thus, the proof of \eqref{kell1} is complete.
%откуда вытекает искомое неравенство \eqref{k-1}.

If $2\le p<\infty$ and $\bigl(\sum_{j\in
\mathbb{Z}}\,\bigl\|\mathcal{F}_{k}(f)\chi_{j}\bigr\|_{p'}^{2}\bigr)^{1/2}<\infty$,
similarly \eqref{dhl-ineq}, we use duality argument to show
\eqref{kell2}.
Indeed, we apply
 Plancherel's theorem, H\"{o}lder's inequality and  \eqref{kell1} to get

%Докажем двойственное \eqref{k-1} неравенство
%\begin{equation}\label{k-2}
%\|f\|_{p}\lesssim \biggl(\sum_{j\in
%\mathbb{Z}}\,\bigl\|\mathcal{F}_{k}(f)\chi_{j}\bigr\|_{p'}^{2}\biggr)^{1/2}.
%\end{equation}
%Действительно, равенство Планшереля и неравенства Гёльдера и \eqref{k-1} (с
%очевидными заменами) для nice functions дают
\begin{align*}
&\quad \int_{\mathbb{R}^{d}}f\overline{g}\,d\mu_{k}=
\int_{\mathbb{R}^{d}}\mathcal{F}_{k}(f)\,\overline{\mathcal{F}_{k}(g)}\,d\mu_{k}=
\sum_{j\in \mathbb{Z}}\int_{\mathbb{R}^{d}}\mathcal{F}_{k}(f)\chi_{j}\,\overline{\mathcal{F}_{k}(g)}\chi_{j}\,d\mu_{k}\\
&\le \sum_{j\in \mathbb{Z}}\,\|\mathcal{F}_{k}(f)\chi_{j}\bigr\|_{p'}
\,\bigl\|\mathcal{F}_{k}(g)\chi_{j}\bigr\|_{p}\le \biggl(\sum_{j\in
\mathbb{Z}}\,\|\mathcal{F}_{k}(f)\chi_{j}\|_{p'}^{2}\biggr)^{1/2} \biggl(\sum_{j\in
\mathbb{Z}}\,\bigl\|\mathcal{F}_{k}(g)\chi_{j}\bigr\|_{p}^{2}\biggr)^{1/2}\\
&\lesssim \biggl(\sum_{j\in
\mathbb{Z}}\,\|\mathcal{F}_{k}(f)\chi_{j}\|_{p'}^{2}\biggr)^{1/2}\|g\|_{p'}
\end{align*}
completing the proof of  \eqref{kell2}.
\end{proof}

\begin{remark}\label{rem-pitt-kell}
Inequalities \eqref{K-1} and \eqref{K-2} easily follow from H\"{o}lder's
inequality for dyadic blocks and the monotonicity of $l_p$-norms.
%Неравенства
%\begin{equation}\label{pk-1}
%\bigl\||x|^{d_{k}(1/p'-1/q)}\mathcal{F}_{k}(f)\bigr\|_{q}\lesssim \biggl(\sum_{j\in
%\mathbb{Z}}\,\bigl\|\mathcal{F}_{k}(f)\chi_{j}\bigr\|_{p'}^{2}\biggr)^{1/2},\quad 2\le q\le
%p',
%\end{equation}
%\begin{equation}\label{pk-2}
%\biggl(\sum_{j\in
%\mathbb{Z}}\,\bigl\|\mathcal{F}_{k}(f)\chi_{j}\bigr\|_{p'}^{2}\biggr)^{1/2}\lesssim
%\bigl\||x|^{d_{k}(1/p'-1/q)}\mathcal{F}_{k}(f)\bigr\|_{q},\quad p'\le q\le 2,
%\end{equation}
%легко следуют из неравенства
% $\bigl(\sum_{j}A_{j}^{\beta}\bigr)^{1/\beta}\le
%\bigl(\sum_{j}A_{j}^{\alpha}\bigr)^{1/\alpha}$ for $0<\alpha\le \beta$.
For example, in order to show \eqref{K-1} with $1<p\le 2\le q\le p'$, we use
$\bigl\||x|^{d_{k}(1/p'-1/q)}g\chi_{j}\bigr\|_{q}\lesssim \|g\chi_{j}\|_{p'}$, $j\in \mathbb{Z}$, to get
\begin{align*}
\biggl(\sum_{j\in
\mathbb{Z}}\,\bigl\|\mathcal{F}_{k}(f)\chi_{j}\bigr\|_{p'}^{2}\biggr)^{1/2}&\gtrsim
\biggl(\sum_{j\in
\mathbb{Z}}\,\bigl\||x|^{d_{k}(1/p'-1/q)}\mathcal{F}_{k}(f)\chi_{j}\bigr\|_{q}^{2}\biggr)^{1/2}\\
&\ge
\biggl(\sum_{j\in
\mathbb{Z}}\,\bigl\||x|^{d_{k}(1/p'-1/q)}\mathcal{F}_{k}(f)\chi_{j}\bigr\|_{q}^{q}\biggr)^{1/q}=
\bigl\||x|^{d_{k}(1/p'-1/q)}\mathcal{F}_{k}(f)\bigr\|_{q}.
\end{align*}

Let us now show that \eqref{K-1} and \eqref{K-2} are sharp. For large enough integer $N$
 %Для доказательства точности \eqref{pk-1}, \eqref{pk-2} для большого целого $N$
 take Schwartz functions  $\psi_{l}$, $l=1,\ldots,N$ such that $\operatorname{supp}\psi_{l}\subset \{2^{l}+\varepsilon\le |x|\le
2^{l}+2\varepsilon\}$ for sufficiently small $\varepsilon>0$ and
$\|\psi_{l}\|_{p'}=1$. Then, by H\"{o}lder's inequality
 %since $\bigl\||x|^{d_{k}(1/p'-1/q)}g\chi_{j}\bigr\|_{q}\lesssim \|g\chi_{j}\|_{p'}$ for  $q\le p'$
   we have
%\begin{align*}
%&\quad \bigl\||x|^{d_{k}(1/p'-1/q)}\psi_{l}\bigr\|_{q}^{q}=\int_{2^{l}\le
%|x|<2^{l+1}}\bigl||x|^{d_{k}(1/p'-1/q)}\psi_{l}\bigr|^{q}\,d\mu_{k}\\
%&\le
%\biggl(\int_{2^{l}\le
%|x|<2^{l+1}}\bigl||x|^{d_{k}(1/p'-1/q)}\bigr|^{q(p'/q)'}\,d\mu_{k}\biggr)^{1/(p'/q)'}
%\biggl(\int_{2^{l}\le
%|x|<2^{l+1}}|\psi_{l}|^{p'}\,d\mu_{k}\biggr)^{q/p'}\\
%&=\biggl(\int_{2^{l}\le
%|x|<2^{l+1}}|x|^{-d_{k}}\,d\mu_{k}\biggr)^{1-q/p'}\|\psi_{l}\|_{p'}^{q}\asymp
%%\biggl(\int_{2^{l}}^{2^{l+1}}\frac{dt}{t}\biggr)^{1-q/p'}\|\psi_{l}\|_{p'}^{q}\asymp
%\|\psi_{l}\|_{p'}^{q}=1.
%\end{align*}
$\bigl\||x|^{d_{k}(1/p'-1/q)}\psi_{l}\bigr\|_{q}\lesssim
\|\psi_{l}\|_{p'}=1$. Similarly,
$\bigl\||x|^{d_{k}(1/p'-1/q)}\psi_{l}\bigr\|_{q}\gtrsim 1$ for $q\ge p'$.

Consider $$f_{N}=\sum_{l=1}^{N}l^{-1/2}\mathcal{F}_{k}^{-1}(\psi_{l}).$$ Since supports  of
$\psi_{l}$
are disjoint sets, we get % and и лежат внутри носителей $\chi_{l}$, поэтому
\[
\biggl(\sum_{j\in
\mathbb{Z}}\,\bigl\|\mathcal{F}_{k}(f_{N})\chi_{j}\bigr\|_{p'}^{2}\biggr)^{1/2}\asymp
\biggl(\sum_{l=1}^{N}l^{-1}\biggr)^{1/2}\asymp (\ln N)^{1/2}.
\]
Thus, for $2\le q\le p'$, % с учетом $\bigl\||x|^{d_{k}(1/p'-1/q)}\psi_{l}\bigr\|_{q}\lesssim 1$
we arrive at
\[
\bigl\||x|^{d_{k}(1/p'-1/q)}\mathcal{F}_{k}(f_{N})\bigr\|_{q}=
\biggl(\sum_{l=1}^{N}l^{-q/2}\bigl\||x|^{d_{k}(1/p'-1/q)}\psi_{l}\bigr\|_{q}^{q}\biggr)^{1/q}\lesssim
\biggl(\sum_{l=1}^{N}l^{-q/2}\biggr)^{1/q}
\]
and the reverse estimate for $p'\le q\le 2$. % имеем обратное неравенство.
This show that for % Таким образом, если
$q\ne 2$, estimates \eqref{K-1} and \eqref{K-2} are optimal. %sharp
%то имеем разный порядок роста по $N$, что доказывает точность \eqref{pk-1}, \eqref{pk-2}.
\qed
\end{remark}

\begin{proof}[Proof of Theorem \ref{thm5}]
To show the estimate
\begin{equation}\label{pitt-omega}
\bigl\||x|^{d_{k}(1/p'-1/q)}\min\{1,(\delta|x|)^{r}\}\mathcal{F}_{k}(f)\bigr\|_{q}\lesssim
\omega_r(f,\delta)_{p},\quad 1<p\le 2,\quad p\le q\le p',
\end{equation}
we use Pitt's inequality \eqref{pitt1} for the difference~\eqref{difference} in
place of $f$. Using also \eqref{transformdifference1}, we have
\[
\bigl\||x|^{d_{k}(1/p'-1/q)}(1-j_{\lambda_{k}}(\delta|x|))^{r/2}\mathcal{F}_{k}(f)\bigr\|_{q}\lesssim
\|\varDelta_\delta^rf\|_{p}.
\]
To conclude the proof, we use \eqref{omega-equiv} and the fact that for
$\lambda>-1/2$ one has
\[
1-j_{\lambda}(t)\asymp \min\{1,t^{2}\}
\]
uniformly in $t\ge 0$. The latter follows from the known properties of the
Bessel function:
%Это вытекает из хорошо известных свойств нормированной функции Бесселя
\[
j_{\lambda}(t)=1-\frac{t^{2}}{4(\lambda+1)}+O(t^{4}),\quad t\to 0,
\]
\[
|j_{\lambda}(t)|\le \min\{1,C_{\lambda}t^{-\lambda-1/2}\},\quad t>0.
\]
%To verify part (2), we use Theorem \ref{thm-pitt} (2)
%\begin{equation}\label{pitt-f-F}
%\|f\|_{p}\lesssim
%\bigl\||x|^{d_{k}(1/p'-1/q)}\mathcal{F}_{k}(f)\bigr\|_{q},
%\end{equation}
%where $2\le p<\infty$, $p'\le q\le p$, and proceed as above.

The reverse inequality to \eqref{pitt-omega} for $2\le p<\infty$, $p'\le q\le p$ can be derived similarly with the help of Pitt's inequality \eqref{pitt2}.

Further, proceeding as in the proof of  \eqref{pitt-omega} with the help of Kellogg's inequality \eqref{kell1},
we establish for $1<p\le 2$
\[
\biggl(\sum_{j\in
\mathbb{Z}}\,\bigl\|\min\{1,(\delta|x|)^{r}\}\mathcal{F}_{k}(f)\chi_{j}\bigr\|_{p'}^{2}\biggr)^{1/2}\lesssim
\omega_r(f,\delta)_{p},
\]
which is equivalent to
%откуда с учетом $\bigl\||x|^{r}\mathcal{F}_{k}(f)\chi_{j}\bigr\|_{p'}\asymp
%2^{j}\bigl\|\mathcal{F}_{k}(f)\chi_{j}\bigr\|_{p'}$ приходим к искомому
%неравенству
\[
\biggl(\sum_{j\in
\mathbb{Z}}\min\{1,(2^{j}\delta)^{2r}\}\bigl\|\mathcal{F}_{k}(f)\chi_{j}\bigr\|_{p'}^{2}\biggr)^{1/2}\lesssim
\omega_r(f,\delta)_{p}.
\]
The case $2\le p<\infty$ is similar. %разбирается аналогично.
\end{proof}

\begin{proof}[Proof of Theorem \ref{thm-besov}]
Relation  \eqref{B-omega} immediately follows from the fact that $\omega_{r}(f,t)_{p}\asymp \omega_{r}(f,2t)_{p}$;
see \eqref{omega-lambda}.

In light of  \eqref{properties_2} and Jackson's inequality  \eqref{eq1}, we derive
\begin{align*}
\sum_{j=1}^{\infty}2^{s\vartheta j} \|\theta_{j}f\|_{p}^\vartheta&\lesssim
\|f\|_p^\vartheta+
\sum_{j=1}^{\infty}2^{s\vartheta j} \|f-\eta_{j}f\|_{p}^\vartheta\lesssim
\|f\|_p^\vartheta+
\sum_{j=0}^{\infty}2^{s\vartheta j} \bigl(E_{2^j}(f)_p\bigr)^\vartheta\\
&\lesssim \|f\|_p^\vartheta+
\sum_{j=0}^{\infty}2^{s\vartheta j}
\bigl(\omega_{r}(f,2^{-j})_{p}\bigr)^{\vartheta}.
\end{align*}
Therefore, to verify  \eqref{B-E},
\eqref{B-eta}, and \eqref{B-theta}, it is enough to show that
\[
\sum_{j=0}^{\infty}2^{s\vartheta j}
\bigl(\omega_{r}(f,2^{-j})_{p}\bigr)^{\vartheta}\lesssim \|f\|_p^{\vartheta}+
\sum_{j=1}^{\infty}2^{s\vartheta j} \|\theta_{j}f\|_{p}^\vartheta.
\]
Using  \eqref{properties_2}, \eqref{properties_3},
\eqref{omega-norma}, and  \eqref{ber-omega} and setting  $\theta_{0}=\eta_{0}$, we have
for $f=\sum_{l=0}^{\infty}\theta_jf$ and $j\ge 0$
\[
\omega_{r}(f,2^{-j})_{p}\lesssim
\sum_{l=0}^{j}\omega_{r}(\theta_lf,2^{-j})_{p}+
\sum_{l=j}^{\infty}\omega_{r}(\theta_lf,2^{-j})_{p}\lesssim
\sum_{l=0}^{j}2^{r(l-j)}\|\theta_lf\|_{p}+
\sum_{l=j}^{\infty}\|\theta_lf\|_{p}.
\]
This and Hardy's inequalities  \eqref{har-1} and \eqref{har-2} with  $r>s$ imply
\begin{align*}
\sum_{j=0}^{\infty}2^{s\vartheta j}
\bigl(\omega_{r}(f,2^{-j})_{p}\bigr)^{\vartheta}&\lesssim
%\sum_{j=0}^{\infty}2^{s\vartheta j}
%\biggl(\sum_{l=0}^{j}2^{r(l-j)}\|\theta_lf\|_{p}\biggr)^{\vartheta}+
%\sum_{j=0}^{\infty}2^{s\vartheta j}
%\biggl(\sum_{l=j}^{\infty}\|\theta_lf\|_{p}\biggr)^{\vartheta}\\
%&=
\sum_{j=0}^{\infty}2^{-j(r-s)\vartheta}
\biggl(\sum_{l=0}^{j}2^{rl}\|\theta_lf\|_{p}\biggr)^{\vartheta}+
\sum_{j=0}^{\infty}2^{s\vartheta j}
\biggl(\sum_{l=j}^{\infty}\|\theta_lf\|_{p}\biggr)^{\vartheta}
\\
&\lesssim
%\sum_{j=0}^{\infty}2^{-j(r-s)\vartheta}
%2^{rj\vartheta}\|\theta_jf\|_{p}^{\vartheta}+
\sum_{j=0}^{\infty}2^{s\vartheta j}\|\theta_jf\|_{p}^{\vartheta} \lesssim \|\eta_{0}f\|_{p}^{\vartheta}+
\sum_{j=1}^{\infty}2^{s\vartheta j}\|\theta_jf\|_{p}^{\vartheta}
\\
&\lesssim
\|f\|_{p}^{\vartheta}+
\sum_{j=1}^{\infty}2^{s\vartheta j}\|\theta_jf\|_{p}^{\vartheta}.
\end{align*}

%Эквивалентности \eqref{B-E}, \eqref{B-eta}, \eqref{B-theta} установлены.
Relation  \eqref{B-B} follows from
\[
\biggl(\sum_{j=-\infty}^{0}2^{s\vartheta j}
\|\theta_{j}f\|_{p}^\vartheta\biggr)^{1/\vartheta}\lesssim
\|f\|_{p}\biggl(\sum_{j=-\infty}^{0}2^{s\vartheta
j}\biggr)^{1/\vartheta}\lesssim \|f\|_{p}.
\]

To obtain  \eqref{B-P}, we take into account  \eqref{B-omega},
Theorem \ref{thm3}, and Hardy's inequalities. In particular, we establish the estimate from above as follows:
\begin{align*}
\sum_{j=0}^{\infty}2^{s\vartheta j}
\bigl(\omega_{r}(f,2^{-j})_{p}\bigr)^{\vartheta}&\lesssim
\sum_{j=0}^{\infty}2^{s\vartheta j}
\biggl(\sum_{l=j+1}^{\infty}2^{-qrl}
\bigl\|(-\Delta_k)^{r/2}P_{l}\bigr\|_{p}^q\biggr)^{\vartheta/q}\\
&\asymp \sum_{j=1}^{\infty}2^{(s-r)\vartheta j}
\bigl\|(-\Delta_k)^{r/2}P_{l}\bigr\|_{p}^{\vartheta}.
\end{align*}
Similarly, we obtain the estimate from below.

Any of equivalences  \eqref{B-E}--\eqref{B-B} shows that the Besov--Dunkl space
equipped with the (quasi-)norm $\|f\|_{B_{p,\vartheta}^{s}}$ does not depend on
the choice of  $r>s$.
\end{proof}

\begin{proof}[Proof of Theorem \ref{thm-besov-pitt}]
%Using notation \eqref{def-I} and
%\[
%F(x)=|x|^{d_{k}(1/p'-1/q)}\mathcal{F}_{k}(f)(x),\quad \chi_{j}=\chi_{\{2^{j}\le
%|x|<2^{j+1}\}},\quad j\in \mathbb{Z},
%\]
We give the proof only for $\vartheta<\infty$. The proof for $\vartheta=\infty$
is similar.

Suppose $2\le p<\infty$ and $\mathcal{F}_{k}(f)\in
L^{p'}(\mathbb{R}^d,d\mu_k)$.
%We first show that $\|f\|_{B_{p,\vartheta}^{s}}\lesssim S(f)$.
First, applying %Theorem~\ref{thm5}~(2),
\eqref{omega-pitt} with $q=p'$, we estimate
\[
\int_{0}^{1}\bigl(t^{-s}\omega_{r}(f,t)_{p}\bigr)^{\vartheta}\,
\frac{dt}{t}\asymp \sum_{j=0}^{\infty}2^{s\vartheta
j}\bigl(\omega_{r}(f,2^{-j})_{p}\bigr)^{\vartheta}\lesssim
\sum_{j=0}^{\infty}2^{s\vartheta
j}\bigl\|\min\{1,(2^{-j}|x|)^{r}\}\mathcal{F}_{k}(f)\bigr\|_{p'}^{\vartheta}.
\]
Second, we note  that
\begin{equation}\label{min-sum}
\sum_{j=0}^{\infty}2^{s\vartheta j}\bigl\|\min\{1,(2^{-j}|x|)^{r}\}\mathcal{F}_{k}(f)\bigr\|_{p'}^{\vartheta}\asymp
\bigl\|\chi_{\{|x|<1\}}|x|^{r}\mathcal{F}_{k}(f)\bigr\|_{p'}+
\biggl(\sum_{j=0}^{\infty}2^{s\vartheta j}\|\mathcal{F}_{k}(f)\chi_{j}\|_{p'}^{\vartheta}\biggr)^{1/\vartheta}.
\end{equation}
Indeed,
\begin{align*}
\sum_{j=0}^{\infty}2^{s\vartheta j}\bigl\|\min\{1,(2^{-j}|x|)^{r}\}\mathcal{F}_{k}(f)\bigr\|_{p'}^{\vartheta}&\asymp
% \sum_{j=0}^{\infty}2^{s\vartheta j}
%\biggl\{\biggl(\int_{|x|<2^{j}}\bigl|\min(1,(2^{-j}|x|)^{r})F(x)\bigr|^{q}\,d\mu_{k}(x)\biggr)^{\vartheta/q}\\
%&\quad {}+\biggl(\int_{|x|\ge
%2^{j}}\bigl|\min(1,(2^{-j}|x|)^{r})F(x)\bigr|^{q}\,d\mu_{k}(x)\biggr)^{\vartheta/q}\biggr\}\\
%&=
\sum_{j=0}^{\infty}2^{(s-r)\vartheta j}
%2^{-jr\vartheta}
\biggl(\int_{|x|<2^{j}}\bigl||x|^{r}\mathcal{F}_{k}(f)\bigr|^{p'}\,d\mu_{k}\biggr)^{\vartheta/p'}\\
&\quad {}+\sum_{j=0}^{\infty}2^{s\vartheta j}
\biggl(\int_{|x|\ge 2^{j}}|\mathcal{F}_{k}(f)|^{p'}\,d\mu_{k}\biggr)^{\vartheta/p'}=I_{1}+I_{2}.
\end{align*}
To estimate $I_{2}$, we use Hardy's inequality \eqref{har-2}:
%\[
%\sum_{j=0}^{\infty}2^{j\gamma}\biggl(\sum_{l=j}^{\infty}A_{l}\biggr)^{\vartheta}\asymp
%\sum_{j=0}^{\infty}2^{j\gamma}A_{j}^{\vartheta}.
%\]Then
\begin{align*}
I_{2}&\asymp
\sum_{j=0}^{\infty}2^{s\vartheta j}
\biggl(\sum_{l=j}^{\infty}\int_{2^{l}\le
|x|<2^{l+1}}|\mathcal{F}_{k}(f)|^{p'}\,d\mu_{k}\biggr)^{\vartheta/p'}\\
&\asymp \sum_{j=0}^{\infty}2^{s\vartheta j}
\biggl(\int_{2^{j}\le
|x|<2^{j+1}}|\mathcal{F}_{k}(f)|^{p'}\,d\mu_{k}\biggr)^{\vartheta/p'}=
\sum_{j=0}^{\infty}2^{s\vartheta j}\|\mathcal{F}_{k}(f)\chi_{j}\|_{p'}^{\vartheta}.
\end{align*}
Taking into account %Для $I_{1}$ по Hardy's inequalities
\eqref{har-1} and $r>s$, we also get
%\[
%\sum_{j=1}^{\infty}2^{-j\gamma}\biggl(\sum_{l=0}^{j-1}A_{l}\biggr)^{\vartheta}\asymp
%\sum_{j=0}^{\infty}2^{-j\gamma}A_{j}^{\vartheta},
%\] учитывая, что $r>s$, находим
\begin{align*}
I_{1}
%&=\sum_{j=0}^{\infty}2^{j(s-r)\vartheta}\biggl(\int_{|x|<2^{j}}
%\bigl||x|^{r}F(x)\bigr|^{q}\,d\mu_{k}(x)\biggr)^{\vartheta/q}
%\\
&\asymp
\sum_{j=0}^{\infty}2^{j(s-r)\vartheta}\biggl(\int_{|x|<1}
\bigl||x|^{r}\mathcal{F}_{k}(f)\bigr|^{p'}\,d\mu_{k}\biggr)^{\vartheta/p'}\\
&\quad {}+\sum_{j=1}^{\infty}2^{j(s-r)\vartheta}\biggl(\int_{1\le |x|<2^{j}}
\bigl||x|^{r}\mathcal{F}_{k}(f)\bigr|^{p'}\,d\mu_{k}\biggr)^{\vartheta/p'}\asymp
\biggl(\int_{|x|<1}\bigl||x|^{r}\mathcal{F}_{k}(f)\bigr|^{p'}\,d\mu_{k}\biggr)^{\vartheta/p'}\\
%\sum_{j=0}^{\infty}2^{j(s-r)\vartheta}
%\\&\quad {}
&\quad {}+
\sum_{j=1}^{\infty}2^{j(s-r)\vartheta}
\biggl(\sum_{l=0}^{j-1}2^{jrp'}\int_{2^{l}\le
|x|<2^{l+1}}|\mathcal{F}_{k}(f)|^{p'}\,d\mu_{k}\biggr)^{\vartheta/p'}
\\
&\asymp
\biggl(\int_{|x|<1}\bigl||x|^{r}\mathcal{F}_{k}(f)\bigr|^{p'}\,d\mu_{k}\biggr)^{\vartheta/p'}+
\sum_{j=0}^{\infty}2^{s\vartheta j}
\biggl(\int_{2^{j}\le |x|<2^{j+1}}|\mathcal{F}_{k}(f)|^{p'}\,d\mu_{k}\biggr)^{\vartheta/p'}.
\end{align*}

Third, in view of  \eqref{min-sum}, we have
\[
\|f\|_{B_{p,\vartheta}^{s}}\lesssim \|f\|_p+
\bigl\|\chi_{\{|x|<1\}}|x|^{r}\mathcal{F}_{k}(f)\bigr\|_{p'}+
\biggl(\sum_{j=0}^{\infty}2^{s\vartheta
j}\|\mathcal{F}_{k}(f)\chi_{j}\|_{p'}^{\vartheta}\biggr)^{1/\vartheta},
\]
where $\bigl\|\chi_{\{|x|<1\}}|x|^{r}\mathcal{F}_{k}(f)\bigr\|_{p'}\le
\|\mathcal{F}_{k}(f)\|_{p'}$ and by
Hausdorff--Young's inequality
$\|f\|_p\lesssim \|\mathcal{F}_{k}(f)\|_{p'}$.
%Pitt's inequality  \eqref{pitt-f-F} yields
Thus,
\[
\|f\|_{B_{p,\vartheta}^{s}}\lesssim
\bigl\|\mathcal{F}_{k}(f)\bigr\|_{p'}+ \biggl(\sum_{j=0}^{\infty}2^{s\vartheta
j} \bigl\|\mathcal{F}_{k}(f)\chi_{j}\bigr\|_{p'}^{\vartheta}\biggr)^{1/\vartheta}.
\]

The reverse estimate %$\|f\|_{B_{p,\vartheta}^{s}}\gtrsim S(f)$
for $1<p\le 2$, $p\le q\le p'$, and $f\in L^p(\mathbb{R}^d,d\mu_k)$ can be
obtained similarly. %using Theorem \ref{thm5} (1).
We have
\[
\int_{0}^{1}\bigl(t^{-s}\omega_{r}(f,t)_{p}\bigr)^{\vartheta}\,
\frac{dt}{t}\gtrsim \sum_{j=0}^{\infty}2^{s\vartheta
j}\bigl\|\min\{1,(2^{-j}|x|)^{r}\}\mathcal{F}_{k}(f)\bigr\|_{p'}^{\vartheta}.
\]
Using \eqref{min-sum},
\[
\|f\|_{B_{p,\vartheta}^{s}}\gtrsim \|f\|_p+
\bigl\|\chi_{\{|x|<1\}}|x|^{r}\mathcal{F}_{k}(f)\bigr\|_{p'}+
\biggl(\sum_{j=0}^{\infty}2^{s\vartheta
j}\|\mathcal{F}_{k}(f)\chi_{j}\|_{p'}^{\vartheta}\biggr)^{1/\vartheta}.
\]
%Pitt's inequality
Hausdorff--Young's inequality $\|f\|_{p}\gtrsim \|\mathcal{F}_{k}(f)\|_{p'}$
completes the proof.
\end{proof}

\end{document}